\documentclass[12pt]{amsart}
\usepackage{amsmath,amssymb}
\usepackage{graphicx,bm}
\usepackage[margin=1in]{geometry}
\usepackage{hyperref}
\usepackage{amsthm}
\usepackage{verbatim}
\usepackage{slashed}
\numberwithin{equation}{section}

\newtheorem{thm}{Theorem}[section]

\newtheorem{cor}[thm]{Corollary}
\newtheorem{lem}[thm]{Lemma}

\newtheorem{defn}[thm]{Definition}
\newtheorem{rmk}[thm]{Remark}

\newcommand{\pt}{\partial}

\begin{document}

\title[$\alpha$-Dirac-harmonic maps from closed surfaces]{$\alpha$-Dirac-harmonic maps from closed surfaces}
\author{J\"urgen Jost, Jingyong Zhu}

\address{Max Planck Institute for Mathematics in the Sciences, Inselstrasse 22, 04103 Leipzig, Germany}
\email{jost@mis.mpg.de}
\address{Max Planck Institute for Mathematics in the Sciences, Inselstrasse 22, 04103 Leipzig, Germany}
\email{jizhu@mis.mpg.de}



\begin{abstract}$\alpha$-Dirac-harmonic maps are variations of Dirac-harmonic maps, analogous to $\alpha$-harmonic maps that were introduced by Sacks-Uhlenbeck to attack the existence problem for harmonic maps from surfaces. For $\alpha >1$, the latter are known to satisfy a Palais-Smale condtion, and so, the technique of Sacks-Uhlenbeck consists in constructing $\alpha$-harmonic maps for $\alpha >1$ and then letting $\alpha \to 1$. The extension of this scheme to Dirac-harmonic maps meets with several difficulties, and in this paper, we start attacking those. We first prove the existence of nontrivial perturbed $\alpha$-Dirac-harmonic maps when the target manifold has nonpositive curvature. The regularity theorem then shows that  they are actually smooth. By $\varepsilon$-regularity and suitable perturbations, we  can then show that such a sequence of perturbed $\alpha$-Dirac-harmonic maps converges to a smooth nontrivial $\alpha$-Dirac-harmonic map. 
\end{abstract}
\maketitle


\tableofcontents
\section{Introduction}
Harmonic maps from closed Riemann surfaces and their variants are important both in mathematics as tools to probe the geometry of a Riemannian manifold and in physics as ground states of the nonlinear sigma model of quantum field theory. They represent a borderline case for the Palais-Smale condition and therefore cannot be directly obtained by standard tools. In \cite{sacks1981existence}, Sacks-Uhlenbeck introduced the notion of $\alpha$-harmonic maps which for $\alpha >1$ (which we shall always assume in this paper) makes the problem subcritical for the Palais-Smale condition. To get the existence of harmonic maps, they consider the convergence of $\alpha$-harmonic maps when $\alpha$ decreases to $1$. In general, there are bubbles (harmonic spheres) preventing the smooth convergence of $\alpha$-harmonic maps. Luckily, a nonpositive curvature condition on the target manifold can exclude such bubbles, and they can therefoe obtain  the existence of harmonic maps into such manifolds in any given homotopy class.

There is another harmonic map type problem that is even more difficult and subtle, even though it has a deep geometric significance. Motivated by the supersymmetric nonlinear sigma model from quantum field theory, see \cite{deligne1999quantum}\cite{jost2009geometry}, Dirac-harmonic maps from spin Riemann surfaces into Riemannian manifolds were introduced in \cite{chen2006dirac}.  Mathematically, they are generalizations of the classical harmonic maps and harmonic spinors. The action functional for Dirac-harmonic maps from spin Riemann surfaces is conformally invariant, like the original action functional for harmonic maps. From the PDE point of view, Dirac-harmonic maps are solutions of an elliptic system consisting of a second order equation and a first order Dirac  equation, and the standard PDE methods do not be directly apply to get the regularity of weak solutions.

In fact, the existence of Dirac-harmonic maps from closed surfaces is a tough problem. Different from the Dirichlet problem (see \cite{jost2017existence} and the references given there), even if there is no bubble, the limit may still be trivial, in the sense that  the spinor part $\psi$ vanishes identically. So far, there are only a few results about the Dirac-harmonic maps from closed surfaces, see  \cite{ammann2013dirac} and \cite{chen2015dirac} for some existence results of uncoupled Dirac-harmonic maps (here uncoupled means that the map part is harmonic) based on  index theory and the Riemann-Roch theorem, respectively. Another possible approach to  this problem uses the  heat flow, like Branding\cite{branding2014evolution} and  Wittmann \cite{wittmann2017short}. Other important approaches come from  critical point theory and homology theory, see Isobe \cite{isobe2012existence}\cite{isobe2017morse}, but so far, this only works in the one-dimensional case, where one speaks of Dirac-geodesics. 

In this paper, we want to start an approach via critical point theory. 
In critical point theory, the Palais-Smale condition is a very strong and useful tool. In general, it fails for the energy functional of harmonic maps from spheres \cite{jost2017riemannian}. Therefore, it is not expected in the case of Dirac-harmonic maps. Actually, even for the $\alpha$-Dirac-harmonic maps, it is still unknown, while it is true in the $\alpha$-harmonic maps case. In \cite{isobe2012existence}, Isobe proved the Palais-Smale condition for the  action functional of nonlinear or perturbed Dirac geodesics, 
\begin{equation}
\mathcal{L}(\phi,\psi)=\frac12\int_{S^1}|\dot{\phi}|^2ds+\frac12\int_{S^1}\langle\psi,\slashed{D}\psi\rangle_{\Sigma S^1\otimes\phi^*TN}ds-\int_{S^1}F(\phi,\psi)ds,
\end{equation}
where $s$ is the angular coordinate on $S^1$, $\dot\phi$ denotes the $s$-derivative of $\phi$, and $F$ is a nonlinear perturbation satisfying some growth and decay conditions with respect to $\psi$.  

Given this state of affairs, in this paper, we shall systematically study critical point theory according to Palais-Smale for $\alpha$-Dirac harmonic maps for $\alpha >1$. Since the natural space on which the variational integral is defined is only a Banach, but not a Hilbert space, we need to develop appropriate Banach space tools, like pseudo-gradients. And since our functionals are not bounded from below, we need to look for critical points other than minima, and therefore, we shall need to carefully investigate the  Palais-Smale condition for our  action functional. Moreover, since variational schemes produce only weak solutions, we shall need to deal with the issue of their regularity. 

In more precise terms, we shall first prove the Palais-Smale condition for the  action functional $\mathcal{L}^\alpha$ of perturbed $\alpha$-Dirac-harmonic maps from closed Riemann surfaces:
\begin{equation}
\mathcal{L}^{\alpha}(\phi,\psi)=\frac12\int_{M}(1+|d\phi|^2)^{\alpha}+\frac12\int_{M}\langle\psi,\slashed{D}\psi\rangle_{\Sigma M\otimes\phi^*TN}-\int_{M}F(\phi,\psi).
\end{equation}
The difficulty is to prove the convergence of the map parts, which is solved by combining the ideas in \cite{urakawa2013calculus} and \cite{isobe2012existence}. With the Palais-Smale condition in hand, we want to prove the existence of nontrivial perturbed $\alpha$-Dirac-harmonic maps. Then an obstacle arises when we deform the configuration space. Since our configuration space is no longer a Hilbert manifold, we have to use the pseudo-gradient flow. However, we need the spinor part of the pseudo-gradient vector field to be of a certain  nice structure. Once having such a special pseudo-gradient vector field, by the deformation lemma, we can prove the following result. As usual in the calculus of variations, we shall need certain standard growth conditions (F1)-(F5) on the nonlinearity $F$; they are\\

\noindent
(F1) \  There exist $p\in(2,4)$ and $C>0$ such that 
\begin{equation*}K\label{f1}
|F_\psi(\phi,\psi)|\leq C(1+|\psi|^{p-1})
\end{equation*}
for any $(\phi,\psi)\in\mathcal{F}^{\alpha,1/2}(M,N)$.

\noindent
(F2) \ There exist $\mu>2$ and $R_1>0$ such that
\begin{equation*}
0<\mu{F}(\phi,\psi)\leq\langle F_{\psi}(\phi,\psi),\psi\rangle
\end{equation*}
for any $(\phi,\psi)\in\mathcal{F}^{\alpha,1/2}(M,N)$ with $|\psi|\geq R_1$.

\noindent
(F3) \  There exist $q<4$ and $C>0$ such that 
\begin{equation*}
F_\phi(\phi,\psi)\leq C(1+|\psi|^q)
\end{equation*}
for any $(\phi,\psi)\in\mathcal{F}^{\alpha,1/2}(M,N)$.

\noindent
(F4) \ For any $(\phi,\psi)\in\mathcal{F}$, we have
\begin{equation*}
 F(\phi,\psi)\geq0.   
\end{equation*}

\noindent
(F5) \ As $|\psi|\to0$, we have 
\begin{equation*}
  F(\phi,\psi)=o(|\psi|^2) \text{ uniformly in }\phi\in N.
\end{equation*}
\begin{thm}\label{thm1.1}
Let $M$ be a closed surface and $N$ a compact manifold. Suppose $F$ satisfies $(\rm F1)-(F5)$ with $\frac{4\alpha}{3\alpha-2}\leq\mu\leq p\leq\frac34\mu+1$ for $\alpha\in(1,2]$. Let $R_1,R_2$ and $\rho$ be as in Lemmas \ref{estimateofa} and \ref{estimateofb}. Then we have $m_\theta<c_\theta<\infty$ and $c_\theta$ is a critical value of $\mathcal{L}^\alpha$ in $\mathcal{F}$.
\end{thm}
 Here $\theta$ is a given homotopy class of maps, $m_\theta$ is the minimizing $\alpha$-energy in $\theta$, $c_\theta$ is defined as:
 \begin{equation}
c_\theta=\inf\limits_{\gamma\in\Gamma(Q_{\theta;R_1,R_2})}\sup\mathcal{L}^\alpha(\gamma(Q_{\theta;R_1,R_2})),
\end{equation}
and $\Gamma(Q_{\theta;R_1,R_2})$ is defined in Definition \ref{Gamma}, $Q_{\theta;R_1,R_2}$ is defined by \eqref{QR1R2}. In particular, when $N$ has nonpositive curvature, our solution is nontrivial.

\begin{cor}\label{cor1.2}
Let $M$ be a closed surface and $N$ a compact Rienmannian manifold with non-positive curvature.
Suppose $F$ satisfies $(\rm F1)-(F5)$ with $\frac{4\alpha}{3\alpha-2}\leq\mu\leq p\leq\frac34\mu+1$ for $\alpha\in(1,2]$. Then for any homotopy class $\theta\in[M,N]$, there exists a non-trivial solution $(\phi,\psi)\in W^{1,2\alpha}(M,N)\times H^{1/2}(M,\Sigma M\otimes\phi^*TN)$ to the perturbed $\alpha$-Dirac-harmonic map equations \eqref{elnonlinearalpha1} and \eqref{elnonlinearalpha2} with $\phi\in\theta$.
\end{cor}

Since $\alpha>1$, $\phi$ in the theorem above is continuous. It is natural to expect the smoothness of  the weak solutions. Due to the perturbation $F$, $F_\psi$ will produce $\|\psi\|^3_{L^4}$ according to the proof in \cite{chen2005regularity}. Therefore, the proof there can not apply to our situation directly. To overcome it, we need to control the $L^\infty$-norm of $\psi$ first. The same phenomenon happens in the proof of the $\varepsilon$-regularity. The following regularity theorem shows that  such a nontrivial solution is actually smooth.

\begin{thm}\label{thm1.3}
Suppose $F\in C^\infty$ satisfies $(\rm F1)$ and $(\rm F3)$ for some $p\leq2+2/\alpha$ and $q\geq0$. Then any weakly perturbed $\alpha$-Dirac-harmonic map is smooth. 
\end{thm}

Now, we get a sequence of perturbed $\alpha$-Dirac-harmonic maps $\{(\phi_k,\psi_k)\}$, which are the critical points of the functionals:
\begin{equation}\label{1.4}
\mathcal{L}_k^{\alpha}(\phi,\psi)=\frac12\int_{M}(1+|d\phi|^2)^{\alpha}+\frac12\int_{M}\langle\psi,\slashed{D}\psi\rangle_{\Sigma M\otimes\phi^*TN}-\frac1k\int_{M}F(\phi,\psi),
\end{equation}
where $F$ satisfies the assumptions in Theorem \ref{cor1.2} and Theorem \ref{thm1.3}. For example, one can just take $F=|\psi|^{\frac{4\alpha}{3\alpha-2}}$. Since the proof of Theorem \ref{thm1.3} depends on the oscillation of $\phi$, which cannot be uniformly controlled for $\phi_k$, we need the $\varepsilon$-regularity theorem. This kind of regularity theorem was introduced by Sacks and Uhlenbeck for the $\alpha$-harmonic maps in \cite{sacks1981existence}. When coupled with the Dirac equation, this was handled \cite{chen2005regularity} for Dirac-harmonic maps and in \cite{jost2017existence}\cite{jost2018geometric} for a sequence of $\alpha$-Dirac-harmonic maps($\alpha\to1$). We have to modify here the  growth conditions (F1) and (F3) for the derivatives of $F$; we need\\

\noindent
$(\rm F6)$ \ $|F_\psi(\phi,\psi)|\leq C|\psi|^{r-1} \ \text{for} \ 3<r\leq2+2/\alpha,$\\
$(\rm F7)$ \ $|F_\phi(\phi,\psi)|\leq C|\psi|^{q} \ \text{for} \ q>2.$

\begin{thm}\label{thm1.4}
Suppose $F$ satisfies {\rm (F6)} and {\rm (F7)}.
There is $\varepsilon_0>0$ and $\alpha_0>0$ such that if $(\phi,\psi):(D,g_{\beta\gamma})\to(N,g_{ij})$ is a smooth perturbed $\alpha$-Dirac-harmonic map satisfying
\begin{equation}\label{epsilon}
\int_M(|d\phi|^{2\alpha}+|\psi|^4)\leq\Lambda<+\infty\ \text{and} \ \int_D|d\phi|^2\leq\varepsilon_0 
\end{equation}
for $1\leq\alpha<\alpha_0$, then we have
\begin{equation}
\|d\phi\|_{\tilde{D},1,p}+\|\psi\|_{\tilde{D},1,p}\leq C(D,N,\Lambda,p)(\|d\phi\|_{D,0,2}+\|\psi\|_{D,0,4}),
\end{equation}
and
\begin{equation}\label{boundpsi}
\|\psi\|_{L^\infty(\tilde{D})}\leq C(D,N,\Lambda)\|\psi\|_{D,0,4}
\end{equation}
for any $\tilde D\subset D$ and $p>1$, where $C(D,N,\Lambda,p)$ denotes a constant depending on $D,N,\Lambda$ and $p$.
\end{thm}

With this  $\varepsilon$-regularity, one can easily  prove 
\begin{thm}\label{thm1.5}
For the $\alpha_0$ given in Theorem \ref{thm1.4} and each $\alpha\in(1,\alpha_0)$, let $(\phi_k,\psi_k)$ be the smooth critical points of the functional $L^\alpha_k$ in \eqref{1.4} with uniformly bounded energy:
\begin{equation}\label{energycondition}
E(\phi_k,\psi_k;M):=\int_M(|d\phi|^{2\alpha}+|\psi|^4)\leq\Lambda<+\infty.
\end{equation} 
Suppose F satisfies $(\rm F6)$ and $(\rm F7)$. Then there exist a subsequence, still denoted by $\{(\phi_k,\psi_k)\}$, and a smooth $\alpha$-Dirac-harmonic map $(\phi,\psi)$ such that
 \begin{equation}\label{strong}
 (\phi_k,\psi_k)\to(\phi,\psi) \ \text{in} \  C_{loc}^\infty(M)
 \end{equation}
\end{thm}

Since the convergence in the theorem above is smooth on all of $M$, the action functional of the $\alpha$-Dirac-harmonic map $(\phi,\psi)$ is strictly bigger than $m_\theta$. Then the convexity of the action functional tells us that $(\phi,\psi)$ is a nontrivial $\alpha$-Dirac-harmonic map. Thus, we obtain:
\begin{cor}\label{cor1.6}
Let $M$ be a closed surface and $N$ a compact Rienmannian manifold with non-positive curvature.
For the $\alpha_0$ given in Theorem \ref{thm1.4} and each $\alpha\in(1,\alpha_0)$, if the sequence of perturbed $\alpha$-Dirac-harmonic maps $\{(\phi_k,\psi_k)\}$ satisfies the uniform bounded energy condition \eqref{energycondition}, then there exists a nontrivial $\alpha$-Dirac-harmonic map from $M$ to $N$ with $\phi$ in the given homotopy class $\theta$.
\end{cor}

\begin{rmk}
To get the existence of a Dirac-harmonic map, it is natural to consider the convergence of a sequence of $\alpha$-Dirac-harmonic maps as $\alpha$ decreases to $1$. By the joint work of the first author with Lei Liu and Miaomiao Zhu in \cite{jost2018geometric}, under the uniform bounded energy condition and no bubble condition, there exists a Dirac-harmonic map with the map part in the given homotopy class. However, in general, we cannot guarantee that the  Dirac-harmonic map is nontrivial. 
\end{rmk}

The rest of the paper is organized as follows: In Section 2, we derive the Euler-Lagrange equations and define the configuration space. In Section 3, we prove the Palais-Smale condition for the action functional of perturbed $\alpha$-Dirac-harmonic maps. In Section 4, we construct a special pseudo-gradient vector field and deform our configuration space by the negative pseudo-gradient flow. Besides, we also recall some facts about the linking geometry. In Section 5, we prove the uniqueness of $\alpha$-harmonic maps under the assumption that the target manifold has nonpositive curvature. In Section 6, we give the proof of Theorem \ref{thm1.1} and Theorem \ref{cor1.2}. In Section 7, we prove  Theorem \ref{thm1.3}. In the last section, we prove  Theorem \ref{thm1.4} and Theorem \ref{thm1.5}.

\section{Euler-Lagrange equations and configuration space}    
\subsection{Euler-Lagrange equations}
Let $(M, g_{\beta\gamma})$ be a compact surface with a fixed spin structure, $\Sigma M$ the spinor bundle. For any $X\in\Gamma(TM)$ and $\xi\in\Gamma(\Sigma M)$, the Clifford multiplication is skew-adjointness:
\begin{equation}
\langle X\cdot\xi, \eta\rangle_{\Sigma M}=-\langle\xi, X\cdot\eta\rangle_{\Sigma M},
\end{equation}
where $\langle\cdot, \cdot\rangle_{\Sigma M}$ denotes the Hermitian inner product induced by the metric $g_{\beta\gamma}$. Choosing a local orthonormal basis $\{e_{\beta}\}_{\beta=1,2}$ on $M$, the usual Dirac operator is defined as $\slashed\partial:=e_\beta\cdot\nabla_\beta$, where $\nabla$ stands for the spin connection on $\Sigma M$ (here and in the sequel, we use the Einstein summation convention). One can find more about spin geometry in \cite{lawson1989spin}.

Let $\phi$ be a smooth map from $M$ to a compact Riemannian manifold $(N, h)$ of dimension $n\geq2$. Let $\phi^*TN$ be the pull-back bundle of $TN$ by $\phi$ and consider the twisted bundle $\Sigma M\otimes\phi^*TN$. On this bundle there is a metric $\langle\cdot,\cdot\rangle_{\Sigma M\otimes\phi^*TN}$ induced from the metric on $\Sigma M$ and $\phi^*TN$. Also, we have a connection $\tilde\nabla$ on this twisted bundle naturally induced from those on $\Sigma M$ and $\phi^*TN$. In local coordinates $\{y^i\}_{i=1,\dots,n}$, the section $\psi$ of $\Sigma M\otimes\phi^*TN$ is written as 
$$\psi=\psi_i\otimes\partial_{y^i}(\phi),$$
where each $\psi^i$ is a usual spinor on $M$. We also have following local expression of $\tilde\nabla$
$$\tilde\nabla\psi=\nabla\psi^i\otimes\partial_{y^i}(\phi)+\Gamma_{jk}^i(\phi)\nabla\phi^j\psi^k\otimes\partial_{y^i}(\phi),$$
where $\Gamma^i_{jk}$ are the Christoffel symbols of the Levi-Civita connection of $N$. The Dirac operator along the map $\phi$ is defined as
\begin{equation}\label{dirac}
\slashed{D}:=e_\alpha\cdot\tilde\nabla_{e_\alpha}\psi=\slashed\partial\psi^i\otimes\partial_{y^i}(\phi)+\Gamma_{jk}^i(\phi)\nabla_{e_\alpha}\phi^j(e_\alpha\cdot\psi^k)\otimes\partial_{y^i}(\phi),
\end{equation}
which is self-adjoint\cite{jost2017riemannian}. Sometimes, we use $\slashed{D}_\phi$ to distinguish the Dirac operators defined on different maps.  \cite{chen2006dirac} introduced the functional
\begin{equation}\begin{split}
L(\phi,\psi)&:=\int_M(|d\phi|^2+\langle\psi,\slashed{D}\psi\rangle_{\Sigma M\otimes\phi^*TN})\\
&=\int_M h_{ij}(\phi)g^{\alpha\beta}\frac{\partial\phi^i}{\partial x^\alpha}\frac{\pt\phi^j}{\pt x^\beta}+h_{ij}(\phi)\langle\psi^i,\slashed{D}\psi^j\rangle_{\Sigma M}.
\end{split}
\end{equation}
with  Euler-Lagrange equations 
\begin{equation}\label{eldh1}
\tau^m(\phi)-\frac12R^m_{lij}\langle\psi^i,\nabla\phi^l\cdot\psi^j\rangle_{\Sigma M}=0,
\end{equation}
\begin{equation}\label{eldh2}
\slashed{D}\psi^i=\slashed\pt\psi^i+\Gamma_{jk}^i(\phi)\nabla_{e_\alpha}\phi^j(e_\alpha\cdot\psi^k)=0,
\end{equation}
where $\tau^m(\phi)$ is the $m$-th component of the tension field\cite{jost2017riemannian} of the map $\phi$ with respect to the coordinate on $N$, $\nabla\phi^l\cdot\psi^j$ denotes the Clifford multiplication of the vector field $\nabla\phi^l$ with the spinor $\psi^j$, and $R^m_{lij}$ stands for the component of the Riemann curvature tensor of the target manifold $N$. Denote $$\mathcal{R}(\phi,\psi):=\frac12R^m_{lij}\langle\psi^i,\nabla\phi^l\cdot\psi^j\rangle_{\Sigma M}\pt_{y^m}$$. We can write \eqref{eldh1} and \eqref{eldh2} in the  global form
\begin{equation}\label{geldh}
\begin{split}
\tau(\phi)&=\mathcal{R}(\phi,\psi)\\
\slashed{D}\psi&=0.
\end{split}
\end{equation}
Solutions $(\phi,\psi)$ of \eqref{geldh} are called Dirac-harmonic maps from $M$ to $N$.

In this paper, we try to apply the critical point theory to the existence problem of Dirac-harmonic maps from closed surfaces. For the one-dimensional case, Takeshi Isobe\cite{isobe2012existence} proves the existence of nontrivial nonlinear Dirac-geodesics on flat tori, which are  critical point of 
\begin{equation}\label{nonlinear1}
\mathcal{L}(\phi,\psi)=\frac12\int_{S^1}|\dot{\phi}|^2ds+\frac12\int_{S^1}\langle\psi,\slashed{D}\psi\rangle_{\Sigma S^1\otimes\phi^*TN}ds-\int_{S^1}F(\phi,\psi)ds,
\end{equation}
where $s$ is the angular coordinate on $S^1$, $\dot\phi$ denotes the $s$-derivative of $\phi$, and $F$ is a nonlinear interaction term satisfying some growth and decay conditions with respect to $\psi$. The Euler-Lagrange equations are:
\begin{equation}\label{elnonlinear}
\begin{split}
\nabla_{\frac{\pt}{\pt s}}\dot\phi-\mathcal{R}(\phi,\psi)+F_\phi(\phi,\psi)&=0,\\
\slashed{D}\psi-F_\psi(\phi,\psi)&=0.
\end{split}
\end{equation}
In particular, Isobe proved that there exists a non-trivial solution $(\phi, \psi)$ to \eqref{elnonlinear} with $\phi$  in any given free homotopy class of loops on a flat torus. In \cite{isobe2017morse}, Isobe reconsidered this problem through homology theory. By constructing and computing a Morse-Floer type homology, he obtains several existence results for perturbed Dirac geodesics. Some of them do not need the curvature restriction on the target manifold.

To generalize Isobe's result to closed surfaces, we need to overcome two obstacles. One is to prove the  Palais-Smale condition in the two dimensional setting, the other is to construct a nice pseudo-gradient vector field. Since the energy functional $E(\phi)=\int_M|d\phi|^2$ in general does not satisfy the Palais-Smale condition  in two dimension, analogously to  $\alpha$-harmonic map\cite{sacks1981existence}, we consider 
\begin{equation}\label{nonlinearalpha}
\mathcal{L}^{\alpha}(\phi,\psi)=\frac12\int_{M}(1+|d\phi|^2)^{\alpha}+\frac12\int_{M}\langle\psi,\slashed{D}\psi\rangle_{\Sigma M\otimes\phi^*TN}-\int_{M}F(\phi,\psi).
\end{equation}
Similar to the computations in \cite{isobe2012existence}, one can get the Euler-Lagrange equations for $\mathcal{L}^{\alpha}$:
\begin{equation}\label{elnonlinearalpha1}
\tau^{\alpha}(\phi):=\tau((1+|d\phi|^2)^\alpha)=\mathcal{R}(\phi,\psi)-F_\phi(\phi,\psi),
\end{equation}
\begin{equation}\label{elnonlinearalpha2}
\slashed{D}\psi=F_{\psi}(\phi,\psi).
\end{equation}

\subsection{Configuration space}
We will define a configuration space for our functional $\mathcal{L}^{\alpha}$ that is a natural modification of the configuration space in \cite{isobe2012existence}. In fact, we focus on   $W^{1,2\alpha}$-maps$(\alpha>1)$ and $H^{1/2}$-spinors. By the Nash embedding theorem\cite{nash1956imbedding}, we can embed $N$ into Euclidean space $\mathbb{R}^L$ for some large $L$. We define the $W^{1,2\alpha}$-maps on $N$ as
\begin{equation}
W^{1,2\alpha}(M,N):=\{\phi\in W^{1,2\alpha}(M,\mathbb{R}^L)| \ \phi(x)\in N\ \  \text{for a.e.} \ x\in M\},
\end{equation}
where $\phi\in W^{1,2\alpha}(M,\mathbb{R}^L)$ means that both $\phi$ and its weak derivative $\nabla\phi$ are in $L^{2\alpha}(M,\mathbb{R}^L)$. By the Sobolev embedding theorem, any $\phi\in W^{1,2\alpha}(M,N)$ is continuous. Therefore, the pull-back bundle $\phi^*TN$ is well-defined, and we can consider $H^{1/2}$-spinors along $\phi\in W^{1,2\alpha}(M,N)$ defined as
\begin{equation}
\begin{split}
H^{1/2}(M,\Sigma M\otimes\phi^*TN):=&\{\psi\in H^{1/2}(M,\Sigma M\otimes\underline{\mathbb{R}}^L)|\\
& \ \psi(x)\in\Sigma M\otimes T_{\phi(x)}N\ \ \text{for a.e.} \ x\in M\},
\end{split}
\end{equation}
where $\underline{\mathbb{R}}^L=M\times\mathbb{R}^L$ is the trivial $\mathbb{R}^L$-bundle over $M$ and we regard $T_{\phi(x)}N$ as a subset of $\mathbb{R}^L$ for each $x\in M$ by the above embedding.

Consider the Banach space $\mathcal{F}^{\alpha,1/2}:=W^{1,2\alpha}(M,\mathbb{R}^L)\times H^{1/2}(M,\Sigma M\otimes\underline{\mathbb{R}}^L)$. Since $\mathcal{F}^{\alpha,1/2}$ is a product space, the norm can be induced from the norms on the two subspaces. We view the map component as the horizontal part and the spinor component as the vertical part. The horizontal part is a Banach space with the usual norm. The vertical part is actually a Hilbert space with the following inner product\cite{ammann2003variational}\cite{gong2017coupled}\cite{isobe2017morse}:
\begin{equation}
(Y_1,Y_2)_{1/2,2}=((1+|{\bf D}|)Y_1,Y_2)_2,
\end{equation}
where $(\cdot,\cdot)_2$ is the $L^2$-inner product on $M$ and ${\bf D}:=\slashed{\pt}\otimes{\bf 1}$. With respect to this product structure, we can write 
\begin{equation}
d\mathcal{L}^\alpha=d^H\mathcal{L}^\alpha+d^V\mathcal{L}^\alpha,
\end{equation}
where 
\begin{equation}
d^H\mathcal{L}^\alpha(\phi,\psi)(X)=(-\tau^\alpha(\phi)+\mathcal{R}(\phi,\psi)-F_\phi(\phi,\psi),X)_2
\end{equation}
and 
\begin{equation}
d^V\mathcal{L}^\alpha(\phi,\psi)(Y)=(\nabla^V\mathcal{L}^\alpha(\phi,\psi),Y)_{1/2,2},
\end{equation}
\begin{equation}
\nabla^V\mathcal{L}^\alpha(\phi,\psi)=(1+|{\bf D}|)^{-1}(\slashed{D}\psi-F_\psi(\phi,\psi)).
\end{equation}

Now, we can define a configuration space $\mathcal{F}^{\alpha,1/2}(M,N)$ as 
\begin{equation}\label{configuration}
\mathcal{F}^{\alpha,1/2}(M,N):=\{(\phi,\psi)| \ \phi\in W^{1,2\alpha}(M,N), \ \psi\in H^{1/2}(M,\Sigma M\otimes\phi^*TN)\},
\end{equation}
which is a Banach submanifold of $\mathcal{F}^{\alpha,1/2}$ with the tangent space at $(\phi,\psi)\in\mathcal{F}^{\alpha,1/2}(M,N)$ being
\begin{equation}\label{tangent}
T_{(\phi,\psi)}\mathcal{F}^{\alpha,1/2}(M,N)=W^{1,2\alpha}(M,\phi^*TN)\times H^{1/2}(M,\Sigma M\otimes\phi^*TN).
\end{equation}
The space $W^{1,2\alpha}(M,N)$ is a Banach manifold\cite{urakawa2013calculus} whose tangent space at $\phi\in W^{1,2\alpha}(M,N)$ is 
\begin{equation}
\begin{split}
T_{\phi}W^{1,2\alpha}(M,N)&=W^{1,2\alpha}(M,\phi^*TN)\\
&:=\{X\in W^{1,2\alpha}(M,\underline{\mathbb{R}}^L)| \ X(x)\in T_{\phi(x)}N\ \ \text{for a.e.} \ x\in M\}.
\end{split}
\end{equation}
So, to see \eqref{tangent}, it suffices to show the vertical part, which is the  same as in \cite{isobe2012existence}.

Let us go back to  our Lagrangian $\mathcal{L}^{\alpha}$. On $\mathcal{F}^{\alpha,1/2}(M,N)$, since $H^{1/2}$ can be continuously embedded into $L^4$\cite{taylor1981pseudodifferential}, $\mathcal{L}^{\alpha}$ is well-defined if for example $F(\phi,\psi)$ grows at most like  $|\psi|^4$ as $|\psi|\to\infty$.

\section{The Palais-Smale condition}
In this section, we prove the Palais-Smale condition for $\mathcal{L}^\alpha$ for a certain class of nonlinearities $F$. We assume in this section that $F\in C^1(\mathcal{F}^{\alpha,1/2}(M,N))$ satisfies the following conditions:

(F1) \  There exist $p\in(2,4)$ and $C>0$ such that 
\begin{equation}\label{f1}
|F_\psi(\phi,\psi)|\leq C(1+|\psi|^{p-1})
\end{equation}
for any $(\phi,\psi)\in\mathcal{F}^{\alpha,1/2}(M,N)$.

(F2) \ There exist $\mu>2$ and $R_1>0$ such that
\begin{equation}\label{f2}
0<\mu{F}(\phi,\psi)\leq\langle F_{\psi}(\phi,\psi),\psi\rangle
\end{equation}
for any $(\phi,\psi)\in\mathcal{F}^{\alpha,1/2}(M,N)$ with $|\psi|\geq R_1$.

(F3) \  There exist $q<4$ and $C>0$ such that 
\begin{equation}\label{f3}
F_\phi(\phi,\psi)\leq C(1+|\psi|^q)
\end{equation}
for any $(\phi,\psi)\in\mathcal{F}^{\alpha,1/2}(M,N)$

The constants $C$ in this paper may different between lines. From now on, we denote $\mathcal{F}^{\alpha,1/2}(M,N)$ by $\mathcal{F}$ for short. Let us recall the  Palais-Smale condition:
\begin{defn}
A sequence $\{(\phi_n,\psi_n)\}\subset\mathcal{F}$ is a Palais-Smale sequence if the following are satisfied:
\begin{itemize}
\item[(1)]   $\sup\limits_{n\geq1}|\mathcal{L}^{\alpha}(\phi_n,\psi_n)|<\infty$,
\item[(2)]  $\|d\mathcal{L}^{\alpha}(\phi_n,\psi_n)\|_{T^*_{(\phi_n,\psi_n)}\mathcal{F}}\to 0$.
\end{itemize}
\end{defn}
We say $\mathcal{L}^\alpha$ satisfies the Palais-Smale condition on $\mathcal{F}$ if any Palais-Smale sequence has a convergent subsequence in $\mathcal{F}$. By Sobolev embedding and conditions (F1) (F3), $d\mathcal{L}^\alpha_{(\phi,\psi)}$ is a bounded linear map on $\mathcal{F}$. Thus, Zorn's proposition (see page 30 in \cite{urakawa2013calculus}) implies $\mathcal{L}^\alpha$ is $C^1$ on  $\mathcal{F}$. Therefore, if $\mathcal{L}^\alpha$ satisfies the Palais-Smale condition on $\mathcal{F}$, and so,  any Palais-Smale sequence has a subsequence converging to a critical point.

The following is the main theorem in this section. The proof for the  vertical part follows the one in \cite{isobe2012existence}. For the horizontal part, the method in \cite{isobe2012existence}  no longer applies, because $d^H\mathcal{L}^\alpha$ cannot be written as a combination of a linear  and a compact operator. Although the proof in \cite{urakawa2013calculus} is not valid for our case either, one estimate there inspired our convergence proof for the  horizontal part. For completeness, we also give the proof for the spinor part.
\begin{thm}\label{ps}
If $F$ satisfies conditions $(\rm F1), (F2), (F3)$ with $\frac{4\alpha}{3\alpha-2}\leq\mu\leq p\leq\frac34\mu+1$ for $\alpha\in(1,2]$, then $\mathcal{L}^\alpha$ satisfies the Palais-Smale condition on $\mathcal{F}$.
\end{thm}
\begin{proof}
We first prove that any Palais-Smale sequence is bounded.

Let $\{(\phi_n,\psi_n)\}\subset\mathcal{F}$ be a Palais-Smale sequence. By the structure on $\mathcal{F}^{\alpha,1/2}$, we have 
\begin{equation}\label{dv}
\begin{split}
|d\mathcal{L}^{\alpha}(\phi_n,\psi_n)(0,\psi_n)|&=|d^V\mathcal{L}^{\alpha}(\phi_n,\psi_n)(\psi_n)|
=|(\nabla^V{\mathcal{L}}^\alpha(\phi_n,\psi_n),\psi_n)_{1/2,2}|\\
&=|(\slashed{D}\psi_n-F_\psi(\phi_n,\psi_n),\psi_n)_2|\leq\|\psi_n\|_{1/2,2},
\end{split}\end{equation}
where we have used $\|d^V\mathcal{L}^{\alpha}(\phi_n,\psi_n)\|\to0$ as $n\to\infty$. This implies
\begin{equation}\begin{split}\label{0,psi}
&2\mathcal{L}^\alpha(\phi_n,\psi_n)-d\mathcal{L}^\alpha(\phi_n,\psi_n)(0,\psi_n)\\
&=\int_M(1+|d\phi_n|^2)^\alpha+\int_M\langle\psi_n,\slashed{D}\psi_n\rangle-2\int_MF(\phi_n,\psi_n)-\int_M\langle\slashed{D}\psi_n-F_\psi(\phi_n,\psi_n),\psi_n\rangle\\
&=\int_M(1+|d\phi_n|^2)^\alpha-2\int_MF(\phi_n,\psi_n)+\int_M\langle F_\psi(\phi_n,\psi_n),\psi_n\rangle\\
&\leq C+\|\psi_n\|_{1/2,2}.
\end{split}\end{equation}
For simplicity, we denote all  positive constants that are independent of $n$ by $C$.

On the other hand, by (F2), there exists a constant $C>0$ such that 
\begin{equation}
\langle F_\psi(\phi,\psi),\psi\rangle\geq\mu{F}(\phi,\psi)-C
\end{equation}
for any $(\phi,\psi)\in\mathcal{F}^{\alpha,1/2}(M,N)$. Plugging this into \eqref{0,psi}, we get
\begin{equation}\label{0,psi2}
\int_M(1+|d\phi_n|^2)^\alpha+(\mu-2)\int_MF(\phi_n,\psi_n)\leq\|\psi_n\|_{1/2,2}+C.
\end{equation}
Integrating (F2), we know
\begin{equation}
F(\phi,\psi)\geq C|\phi|^\mu-C
\end{equation}
for any $(\phi,\psi)\in\mathcal{F}^{\alpha,1/2}(M,N)$. Plugging this into \eqref{0,psi2} yields
\begin{equation}\label{munorm}
\|\phi_n\|^{2\alpha}_{1,2\alpha}+\|\psi\|^\mu_\mu\leq C(\|\psi_n\|_{1/2,2}+1).
\end{equation}

Consider the spectral decomposition
\begin{equation}
H^{1/2}(M,\Sigma M\otimes\underline{\mathbb{R}}^L)={\bf H}_0^-\oplus{\bf H}_0^0\oplus{\bf H}_0^+
\end{equation}
with respect to the operator ${\bf D}:=\slashed{\pt}\otimes{\bf 1}$, where ${\bf H}_0^-$, ${\bf H}_0^0$ and ${\bf H}_0^+$ are the closures in $H^{1/2}(M,\Sigma M\otimes\underline{\mathbb{R}}^L)$ of the spaces spanned by the negative, the null and the positive eigenspinors of ${\bf D}$, respectively. Denote by ${\rm P}_0^-: H^{1/2}(M,\Sigma M\otimes\underline{\mathbb{R}}^L)\to{\bf H}_0^-$, ${\rm P}_0^0: H^{1/2}(M,\Sigma M\otimes\underline{\mathbb{R}}^L)\to{\bf H}_0^0$ and ${\rm P}_0^+: H^{1/2}(M,\Sigma M\otimes\underline{\mathbb{R}}^L)\to{\bf H}_0^+$ the corresponding spectral projections.
By this decomposition, we write $\phi_n=\psi_{n,0}^-+\psi_{n,0}^0+\psi_{n,0}^+$ with $\psi_{n,0}^-$, $\psi_{n,0}^0$ and $\psi_{n,0}^+$ being ${\rm P}_0^-\psi_n$, ${\rm P}_0^0\psi_n$ and ${\rm P}_0^+\psi_n$, respectively. The square of the  $H^{1/2}$- norm of $\psi_{n,0}^+\in{\bf H}_0^+$ is equivalent to (\cite{ammann2003variational})
\begin{equation}
\int_M\langle\psi_{n,0}^+,|{\bf D}|\psi_{n,0}^+\rangle+\|\psi_{n,0}^+\|_2^2=\int_M\langle\psi_{n,0}^+,{\bf D}\psi_{n,0}^+\rangle+\|\psi_{n,0}^+\|_2^2.
\end{equation}

Together with \eqref{dirac}, this implies
\begin{equation}\label{spinornorm}
\begin{split}
\|\psi_{n,0}^+\|^2_{1/2,2}
&\leq C\int_M\langle\psi_{n,0}^+,{\bf D}\psi_n\rangle+C\|\psi_{n,0}^+\|_2^2\\
&\leq C\Big|\int_M\langle\psi_{n,0}^+,{\slashed D}\psi_n\rangle\Big|+C\int_M|d\phi_n||\psi_{n,0}^+||\psi_n|+C\|\psi_n\|_2^2\\
&\leq C\Big|\int_M\langle\psi_{n,0}^+,{\slashed D}\psi_n\rangle\Big|+C\|\phi_n\|_{1,2\alpha}\|\psi_n\|_4\|\psi_{n,0}^+\|_{4\alpha/(3\alpha-2)}+C\|\psi_n\|_2^2\\
&\leq C\Big|\int_M\langle\psi_{n,0}^+,{\slashed D}\psi_n\rangle\Big|+C\|\phi_n\|_{1,2\alpha}\|\psi_n\|_{1/2,2}\|\psi_n\|_{4\alpha/(3\alpha-2)}+C\|\psi_n\|_2^2,
\end{split}\end{equation}
where we have used the H\"older inequality and the Sobolev embedding $H^{1/2}\subset L^4$ for surfaces\cite{taylor1981pseudodifferential}.

Again,  as in \eqref{dv}, we have
\begin{equation}
d\mathcal{L}^{\alpha}(\phi_n,\psi_n)(0,\psi_{n,0}^+)|=|\int_M\langle\slashed{D}\psi_n-F_\psi(\phi_n,\psi_n),\psi_{n,0}^+\rangle|\leq\|\psi_{n,0}^+\|_{1/2,2}.
\end{equation}
This and (F1) give us
\begin{equation}\label{mainterm}
\begin{split}
\Big|\int_M\langle\psi_{n,0}^+,{\slashed D}\psi_n\rangle\Big|&\leq\int_M|F_\psi(\phi_n,\psi_n)||\psi_{n,0}^+|+\|\psi_{n,0}^+\|_{1/2,2}\\
&\leq C\int_M(1+|\psi_n|^{p-1})|\psi_{n,0}^+|+\|\psi_n\|_{1/2,2}\\
&\leq C\|\psi_n\|^{p-1}_\mu\|\psi_{n,0}^+\|_{\mu/(\mu-p+1)}+C\|\psi_n\|_{1/2,2}\\
&\leq C\|\psi_n\|^{p-1}_\mu\|\psi_n\|_{1/2,2}+C\|\psi_n\|_{1/2,2},
\end{split}\end{equation}
where we have used the H\"older inequality and assumption $p\leq\frac34\mu+1$.

Now, plugging \eqref{mainterm} into \eqref{spinornorm}, we have
\begin{equation}\label{spinornorm2}
\begin{split}
\|\psi_{n,0}^+\|^2_{1/2,2}\leq&C(\|\psi_n\|^{p-1}_\mu\|\psi_n\|_{1/2,2}+\|\phi_n\|_{1,2\alpha}\|\psi_n\|_{1/2,2}\|\psi_n\|_{4\alpha/(3\alpha-2)}\\
&+\|\psi_n\|_2^2+\|\psi_n\|_{1/2,2}).
\end{split}
\end{equation}
Since $\alpha\in(1,2]$, the H\"older inequality and \eqref{munorm} imply
\begin{equation}\label{l2norm}
 \|\psi_n\|_2\leq C\|\psi_n\|_{4\alpha/(3\alpha-2)}\leq C\|\psi_n\|_{\mu}\leq C(\|\psi_n\|_{1/2,2}^{1/\mu}+1)
 \end{equation}
 and 
\begin{equation}\label{mapnorm}
\|\phi_n\|_{1,2\alpha}\leq C(\|\psi_n\|_{1/2,2}^{1/2\alpha}+1).
 \end{equation}
Plugging \eqref{l2norm} and \eqref{mapnorm} into \eqref{spinornorm2}, we get
\begin{equation}\label{spinornorm+}
\begin{split}
\|\psi_{n,0}^+\|^2_{1/2,2}\leq&C\|\psi_n\|_{1/2,2}^{(p-1)/\mu}\|\psi_n\|_{1/2,2}+C(\|\psi_n\|_{1/2,2}^{1/2\alpha}+1)\|\psi_n\|_{1/2,2}(\|\psi_n\|_{1/2,2}^{1/\mu}+1)\\
&+C\|\psi_n\|_{1/2,2}^{2/\mu}+C\|\psi_n\|_{1/2,2}+C.\\
\end{split}
\end{equation}
Similarly,
\begin{equation}\label{spinornorm-}
\begin{split}
\|\psi_{n,0}^-\|^2_{1/2,2}\leq&C\|\psi_n\|_{1/2,2}^{(p-1)/\mu}\|\psi_n\|_{1/2,2}+C(\|\psi_n\|_{1/2,2}^{1/2\alpha}+1)\|\psi_n\|_{1/2,2}(\|\psi_n\|_{1/2,2}^{1/\mu}+1)\\
&+C\|\psi_n\|_{1/2,2}^{2/\mu}+C\|\psi_n\|_{1/2,2}+C.\\
\end{split}
\end{equation}
Since the usual Dirac operator has finite dimensional kernel, ${\rm dim}({\bf H}_0^0)<\infty$. Noting that the $H^{1/2}$-norm and the $L^2$-norm are equivalent on ${\bf H}_0^0$ (\cite{ammann2003variational}), we have
\begin{equation}\label{spinornorm0}
\|\psi_{n,0}^-\|^2_{1/2,2}\leq C\|\psi_{n,0}^-\|^2_2\leq C\|\psi_n\|^2_2\leq C(\|\psi_n\|_{1/2,2}^{2/\mu}+1).
\end{equation}
Combining \eqref{spinornorm+}, \eqref{spinornorm-} and \eqref{spinornorm0}, we obtain
\begin{equation}\label{spinornormfinal}
\begin{split}
\|\psi_n\|^2_{1/2,2}\leq&C\|\psi_n\|_{1/2,2}^{(p-1)/\mu}\|\psi_n\|_{1/2,2}+C(\|\psi_n\|_{1/2,2}^{1/2\alpha}+1)\|\psi_n\|_{1/2,2}(\|\psi_n\|_{1/2,2}^{1/\mu}+1)\\
&+C\|\psi_n\|_{1/2,2}^{2/\mu}+C\|\psi_n\|_{1/2,2}+C.\\
\end{split}
\end{equation}
Since $\frac{4\alpha}{3\alpha-2}\leq\mu\leq p\leq\frac34\mu+1<\mu+1$ and $\alpha\in(1,2]$, \eqref{spinornormfinal} tells us that $\{\psi_n\}$ is uniformly bounded in $H^{1/2}(M,\Sigma M\otimes\underline{\mathbb{R}}^L)$. Therefore, from \eqref{mapnorm}, $\{\phi_n\}$ is also bounded. Thus, $\{\phi_n,\psi_n\}$ is a bounded sequence.

Next, we show $\{\psi_n\}$ has a convergent subsequence.
 As in \cite{isobe2012existence}, we write the vertical gradient of $\mathcal{L}^\alpha$ as
 \begin{equation}\label{decomposition1}
 \nabla^V\mathcal{L}^\alpha(\phi,\psi)=L^V\psi+K^V_1(\phi,\psi)+K^V_2(\phi,\psi),
 \end{equation}
where 
\begin{equation}
L^V=(1+|{\bf D}|)^{-1}{\bf D}: H^{1/2}(M,\Sigma M\otimes\underline{\mathbb{R}}^L)\to H^{1/2}(M,\Sigma M\otimes\underline{\mathbb{R}}^L)
\end{equation}
is bounded linear and 
\begin{equation}
K_1^V(\phi,\psi)=(1+|{\bf D}|)^{-1}(\Gamma_{jk}^i(\phi)\nabla_{e_\alpha}\phi^j(e_\alpha\cdot\psi^k)\otimes\partial_{y^i}(\phi)),
\end{equation}
\begin{equation}
K_2^V(\phi,\psi)=-(1+|{\bf D}|)^{-1}F_\psi(\phi,\psi).
\end{equation}
Because both $|\psi\nabla\phi|$ and $F_\psi(\phi,\psi)$ belong in $L^r$ with $r>4/3$, both $K^V_1$ and $K^V_2$ are compact. We write $\psi_n=\psi_{n,0}^-+\psi_{n,0}^0+\psi_{n,0}^+$. Since $L^V$ is invertible on ${\bf H}_0^{\pm}$ and $ \nabla^V\mathcal{L}^\alpha(\phi_n,\psi_n)\to0$, we have
\begin{equation}
\psi_{n,0}^\pm=o(1)-(L^V)^{-1}K(\phi_n,\psi_{n,0}^\pm),
\end{equation}
where $o(1)\to0$ as $n\to\infty$. From the boundedness of $\{(\phi_n,\psi_n)\}$ and the compactness of  $(L^V)^{-1}K$ on ${\bf H}_0^{\pm}$, we know $\psi_{n,0}^\pm$ has a convergent subsequence. Again, since ${\rm dim}{\bf H}_0^0<\infty$, $\psi_{n,0}^0$ also has a convergent subsequence. Thus, $\{\psi_n\}$ has a convergent subsequence.

Last, for the convergence of $\{\phi_n\}$, we consider the $\alpha$-energy functional 
\begin{equation}\label{alphaenergy}
E^\alpha(\phi):=\int_Me^\alpha(d\phi):=\frac12\int_M(1+|d\phi|^2)^\alpha.
\end{equation}
for any $\phi\in W^{1,2\alpha}(M,\mathbb{R}^L)$. Then the second derivative of $e^\alpha$ at $d\phi$ with respect to the direction $d\varphi$ is 
\begin{equation}\label{2ndderivative}
\begin{split}
d^2e^\alpha_{d\phi}(d\varphi,d\varphi)&=\frac12\frac{d^2}{dt^2}\Bigg|_{t=0}(1+\langle d\phi+td\varphi,d\phi+td\varphi\rangle)^\alpha\\
&=2\alpha(\alpha-1)(1+|d\phi|^2)^{\alpha-2}\langle d\phi,d\varphi\rangle^2+\alpha(1+|d\phi|^2)^{\alpha-1}|d\varphi|^2\\
&\geq\alpha|d\phi|^{2\alpha-2}|d\varphi|^2.
\end{split}
\end{equation}
For any $\phi_i,\phi_j\in W^{1,2\alpha}(M,\mathbb{R}^L)$, we define
\begin{equation}
c(t):=b+t(a-b),
\end{equation}
where $a:=d\phi_i$ and $b:=d\phi_j$. Then \eqref{2ndderivative} implies
\begin{equation}
d^2e^\alpha_{c(t)}(a-b,a-b)\geq\alpha|c(t)|^{2\alpha-2}|a-b|^2.
\end{equation}
We also have
\begin{equation}
dE^\alpha_\phi(\varphi)=\frac{d}{dt}\Bigg|_{t=0}\int_Me^\alpha(d\phi+td\varphi)=\int_Mde^\alpha_{d\phi}(d\varphi)
\end{equation}
and 
\begin{equation}
\int_0^1d^2e^\alpha_{c(t)}(a-b,a-b)=de^\alpha_a(a-b)-de^\alpha_b(a-b).
\end{equation}
Now, we can control $\|d\phi_i-d\phi_j\|_{2\alpha}$ as follows.
\begin{equation}\label{gradient2alphanorm}
\begin{split}
(dE^\alpha_{\phi_i}-dE^\alpha_{\phi_j})(\phi_i-\phi_j)&=\int_M(de^\alpha_{d\phi_i}-de^\alpha_{d\phi_j})(d\phi_i-d\phi_j)\\
&\int_M\int_0^1d^2e^\alpha_{c(t)}(d\phi_i-d\phi_j,d\phi_i-d\phi_j)\\
&\geq\int_M\int_0^1\alpha|c(t)|^{2\alpha-2}|d\phi_i-d\phi_j|^2\\
&\geq C\|d\phi_i-d\phi_j\|^{2\alpha}_{2\alpha},
 \end{split}
\end{equation}
where the last inequality comes from the Sublemma 3.18 in \cite{urakawa2013calculus} stated as follows:
\begin{lem}
Let $V$ be a finite dimensional real vector space with a norm $|\cdot|$. Let $m$ be a positive integer. Then there exists a positive constant $C$ such that 
\begin{equation}
\int_0^1|x+ty|^mdt\geq C|y|^m
\end{equation}
for any $x,y\in V$.
\end{lem}

To estimate the left-hand side of \eqref{gradient2alphanorm}, we decompose $d^H\mathcal{L}^\alpha$ as
\begin{equation}
d^H\mathcal{L}^\alpha(\phi,\psi)=J\phi+K^H_1(\phi,\psi)+K^H_2(\phi,\psi)+K^H_3(\phi,\psi),
\end{equation}
where
\begin{equation}
K^H_1(\phi,\psi)=-\Gamma^m_{ij}\phi^i_\beta\phi^j_{\gamma}g^{\beta\gamma}(1+|d\phi|^2)^{\alpha-1}\frac{\pt}{\pt y^m},
\end{equation}
\begin{equation}
K^H_2(\phi,\psi)=\frac{1}{2\alpha}R^m_{jkl}\langle\psi^k,\nabla\phi^j\cdot\psi^l\rangle\frac{\pt}{\pt y^m},
\end{equation}
\begin{equation}
K^H_3(\phi,\psi)=-\frac{1}{\alpha}F_\phi(\phi,\psi)
\end{equation}
and $J: W^{1,2\alpha}(M,\mathbb{R}^L)\to(W^{1,2\alpha}(M,\mathbb{R}^L))^*=W^{-1,2\alpha/(2\alpha-1)}(M,\mathbb{R}^L)$
\begin{equation}
J\phi=-{\rm div}((1+|d\phi|^2)^{\alpha-1}\nabla\phi).
\end{equation}
By Sobolev embedding, $K^H=K^H_1+K^H_2+K^H_3:\mathcal{F}\to(W^{1,2\alpha}(M,\mathbb{R}^L))^*$ is compact. Again, we can write 
\begin{equation}
J=o(1)-K^H.
\end{equation}
The boundedness of $\{(\phi_n,\psi_n)\}$ and the compactness of $K^H$ imply that $J(\phi_n)$ has a convergent subsequence. In particular, for any $\varepsilon$, there exists an integer $N_\varepsilon$ such that 
\begin{equation}
|(J(\phi_i)-J(\phi_j))(\phi_i-\phi_j)|\leq\varepsilon\|\phi_i-\phi_j\|_{1,2\alpha}
\end{equation}
for $i,j>N_\varepsilon$. Together this with \eqref{gradient2alphanorm} and
\begin{equation}
(dE^\alpha_{\phi_i}-dE^\alpha_{\phi_j})(\phi_i-\phi_j)=(J(\phi_i)-J(\phi_j))(\phi_i-\phi_j),
\end{equation}
we get
\begin{equation}
\|d\phi_i-d\phi_j\|_{2\alpha}\leq(C\varepsilon\|\phi_i-\phi_j\|_{1,2\alpha})^{\frac{1}{2\alpha}}\leq(C\varepsilon)^{\frac{1}{2\alpha}},
\end{equation}
where we have used the boundedness of $\{\phi_n\}$. Therefore, $\{d\phi_n\}$ contains a Cauchy subsequence in $L^{2\alpha}$. Hence, by Sobolev embedding, $\{\phi_n\}$ has a convergent subsequence.

\end{proof}

\section{Negative pseudo-gradient flow and linking geometry}
\subsection{Negative pseudo-gradient flow}
In this section, we want to deform our configuration space $\mathcal{F}^{\alpha,1/2}(M, N)$. On a Hilbert manifold, the differential and the gradient are equivalent to each other. However, this fails on Banach manifolds, and one has to  use a pseudo-gradient flow instead. In general, there alway exists a pseudo-gradient vector field for a $C^1$ functional, but there  may  be no  explicit expression. To get some nice properties, we need to construct a special pseudo-gradient vector field for $\mathcal{L}^\alpha$. Precisely, we want a pseudo-gradient vector field with the vertical part being the vertical gradient $\nabla^V\mathcal{L}^\alpha$. Before we do this, let us recall the definition of a pseudo-gradient vector.
\begin{defn}\cite{palais1970critical}\label{pg}
Let $M$ be a $C^{r+1}$($r\geq1$) Finsler manifold and $f: M\to\mathbb{R}$ be a $C^1$ function. A vector $X\in T_pM$ is called a pseudo-gradient vector for $f$ at $p$ if $X$ satisfies\\
(i)   $\|X\|\leq2\|df_p\|$,\\
(ii)  $df_p(X)\geq\|df\|^2$.
\end{defn}
A vector field is called a pseudo-gradient vector field for $f$ if at each point of its domain it is a pseudo-gradient vector for $f$. It is well-known that 
\begin{lem}\cite{palais1970critical}
There exists a locally Lipschitz pseudo-gradient vector field for $f$ on $M^*$.
\end{lem}

On $\mathcal{F}^{\alpha,1/2}(M, N)$, we denote the set of regular points of $\mathcal{L}^\alpha$ by 
\begin{equation}
\tilde{\mathcal F}:=\{(\phi,\psi)\in\mathcal{F}^{\alpha,1/2}(M, N)| \ d\mathcal{L}^\alpha_{(\phi,\psi)}\neq0\}.
\end{equation}
The main result in this section is
\begin{thm}\label{pgvf}
Suppose $F\in C^{1,1}_{loc}$ in the fiber direction. Then there exists a locally Lipschitz pseudo-gradient vector field $\omega$ for $\mathcal{L}^\alpha$ on $\tilde{\mathcal F}$ of the  form of 
$
\omega=X\oplus a\nabla^V\mathcal{L}^\alpha
$
for some vector field $X$ and a constant $a\in(1,2)$.
\end{thm}
\begin{proof}
We divide $\tilde{\mathcal F}$ into two subsets
\begin{equation}
\begin{split}
&A=\{(\phi,\psi)\in\mathcal{F}^{\alpha,1/2}(M, N)| \ d^H\mathcal{L}^\alpha=0, \ d^V\mathcal{L}^\alpha\neq0\},\\
&B=\{(\phi,\psi)\in\mathcal{F}^{\alpha,1/2}(M, N)| \ d^H\mathcal{L}^\alpha\neq0\}.
\end{split}
\end{equation}
For any $a\in(1,2)$, $a\nabla^V\mathcal{L}^\alpha$ always satisfies $(i)$ and $(ii)$ in Definition \ref{pg} for $d^V\mathcal{L}^\alpha$. Our assumption implies that $a\nabla^V\mathcal{L}^\alpha$ is locally Lipschitz. Therefore, it suffices to show that there is a locally Lipschitz vector field that satisfies $(i)$ and $(ii)$ in Definition \ref{pg} for $d^H\mathcal{L}^\alpha$.

Over $B$, since $d^H\mathcal{L}\neq0$, we have a vector $\tilde X_\phi$ for each $(\phi,\psi)\in B$ such that 
$\|\tilde X_\phi\|=1$ and $d^H\mathcal{L}^\alpha_{(\phi,\psi)}(\tilde X_\phi)>\frac23\|d^H\mathcal{L}^\alpha_{(\phi,\psi)}\|$. Then, set $X_\phi=\frac32\|d^H\mathcal{L}^\alpha_{(\phi,\psi)}\|\tilde X_\phi$. So $X_\phi$ satisfies Definition \ref{pg} for $d^H\mathcal{L}^\alpha$ with strict inequalities in $(i)$ and $(ii)$. Then extend $X_\phi$ to be a $C^1$ vector field in a neighborhood of $(\phi,\psi)$ (say by making it ``constant" with respect to a chart at $(\phi,\psi)$)\cite{palais1970critical}. Therefore, for each point $(\phi,\psi)\in B$, there is a $C^1$ pseudo-gradient vector field $\omega=X\oplus a\nabla^V\mathcal{L}^\alpha$ for $\mathcal{L}^\alpha$ defined in some neighborhood of $(\phi,\psi)$.

For $A$, we just use the $C^1$ vector field $\omega=0\oplus a\nabla^V\mathcal{L}^\alpha$. Then, for any point in $\tilde{\mathcal F}$, we have
\begin{equation}
\|\omega\|=a\|\nabla^V\mathcal{L}^\alpha\|<2\|\nabla^V\mathcal{L}^\alpha\|\leq2\|d\mathcal{L}^\alpha\|
\end{equation}
and 
\begin{equation}
d\mathcal{L}^\alpha_{(\phi,\psi)}(\omega)=a\|\nabla^V\mathcal{L}^\alpha(\phi,\psi)\|^2.
\end{equation}
So, $\omega$ satisfies $(i)$ in Definition \ref{pg}. To check $(ii)$, we take any point $(\phi,\psi)\in A$. Then
\begin{equation}
\|d\mathcal{L}^\alpha_{(\phi,\psi)}\|^2=\|\nabla^V\mathcal{L}^\alpha(\phi,\psi)\|^2<a\|\nabla^V\mathcal{L}^\alpha(\phi,\psi)\|^2.
\end{equation}
Together with $\mathcal{L}^\alpha\in C^1$, this implies that 
\begin{equation}
d\mathcal{L}^\alpha_{(\phi',\psi')}(\omega)=a\|\nabla^V\mathcal{L}^\alpha(\phi',\psi')\|^2\geq\|d\mathcal{L}^\alpha_{(\phi',\psi')}\|^2
\end{equation}
holds in some neighborhood $U$ of $(\phi,\psi)$. So,  $\omega$ also satisfies $(ii)$ in Definition \ref{pg}. Thus, for each point $(\phi,\psi)\in A$, there is a neighborhood $U$ of $(\phi,\psi)$ such that 
$\omega=0\oplus a\nabla^V\mathcal{L}^\alpha$ is a pseudo-gradient vector field for $\mathcal{L}^\alpha$ in $U$. Finally, to get the pseudo-gradient vector field in the theorem, we can patch these local pseudo-gradient vector fields together by a partition of unity\cite{palais1970critical}\cite{rabinowitz1986minimax}. Thus, we complete the proof.

\end{proof}

Theorem \ref{pgvf} gives us a nice pseudo-gradient vector field, but it is only locally Lipschitz. Therefore, its integral curve may not exist globally. To remedy this, it suffices to integrate a truncated  pseudo-gradient vector field. The argument can be found in \cite{rabinowitz1986minimax}. Different from our case, Isobe \cite{isobe2012existence} used it directly on the gradient of the action functional. Now, we deform our configuration space by integrating the following ODE:
\begin{equation}\label{pgf}
\frac{d}{dt}(\phi_t,\psi_t)=-\eta(\phi_t,\psi_t)\omega(\phi_t,\psi_t),
\end{equation}
\begin{equation}\label{initial}
(\phi_t,\psi_t)|_{t=0}=(\phi,\psi).
\end{equation}
The function $\eta$ is chosen such that \eqref{pgf} and \eqref{initial} have a global unique solution. In particular, $\eta$ satisfies $\eta(\phi,\psi)\|\omega(\phi,\psi)\|\leq1$ for all $(\phi,\psi)\in\mathcal{F}^{\alpha,1/2}(M,N)$. See Appendix A in \cite{rabinowitz1986minimax}.

Note that the solution $\psi_t$ to \eqref{pgf} belongs to $H^{1/2}(M,\Sigma M\otimes\phi_t^*TN)$ for each $t\geq0$ and the space depends on $t$. So we translate it into a flow on a function space which does not depend on $t$. To do so, we consider the parallel transport. For each $x\in M$, we denote by ${\bf P}_t(x): T_{\phi(x)}N\to T_{\phi_t(x)}N$ the parallel transport along the curve $t\mapsto\phi_t(x)$. We put $\tilde\psi_t(x):=({\bf1}\otimes{{\bf P}_t(x)}^{-1})\psi_t(x)$, that is, for $\psi_t(x)=\psi_t^i(x)\otimes\frac{\pt}{\pt y^i}(\phi_t(x))$, we define
\begin{equation}
\tilde\psi_t(x):=\psi^i_t(x)\otimes{{\bf P}_t(x)}^{-1}\Bigg(\frac{\pt}{\pt y^i}(\phi_t(x))\Bigg)\in\Sigma M\otimes\phi^*TN.
\end{equation}
Thus, the vertical part of \eqref{pgf} is transformed to 
\begin{equation}\label{modifiedpgf}
\frac{d}{dt}\tilde\psi_t=-a\eta_t(\tilde\psi_t){\bf P}_t^{-1}\nabla^V\mathcal{L}^\alpha(\phi_t,{\bf P}_t\tilde\psi_t),
\end{equation}
where $\eta_t(\tilde\psi_t)=\eta(\phi_t,\psi_t)$ and $a$ is the constant in Theorem \ref{pgvf}. 

Since the proofs of Lemma 7.2-Lemma 7.6 in \cite{isobe2012existence} only rely on the Sobolev embedding and the vertical part of the gradient vector field, which is kept up to a constant in our case, these nice properties (Lemma 7.2-Lemma 7.6 in \cite{isobe2012existence})  generalize to our configuration space $\mathcal{F}^{\alpha,1/2}(M,N)$. Due to their usefulness, we list them here and refer to \cite{isobe2012existence} for the proofs.
\begin{lem}
Let $(\phi_t,\psi_t)$ and ${\bf P}_t$  be as above. ${\bf P}_t$ defines a bounded linear map ${\bf P}_t={\bf1}\otimes{\bf P}_t: H^{1/2}(M,\Sigma M\otimes\phi^*TN)\to H^{1/2}(M,\Sigma M\otimes\phi_t^*TN)\subset H^{1/2}(M,\Sigma M\otimes\underline{\mathbb{R}}^L)$ which depends coutinuously on $t$ with respect to the operator norm.
\end{lem}


\begin{lem}\label{T}
Let $\phi,\tilde\phi\in W^{1,2\alpha}(M,N)$ be such that $\|\phi-\tilde\phi\|_\infty<\iota(N)$. We define $T_{\phi,\tilde\phi}(x):\Sigma M\otimes T_{\phi(x)}N\to\Sigma M\otimes T_{\tilde\phi(x)}N$ by $T_{\phi,\tilde\phi}(x)={\bf1}\otimes{\bf P}_{\phi(x),\tilde\phi(x)}$, where ${\bf P}_{\phi(x),\tilde\phi(x)}:T_{\phi(x)}N\to T_{\tilde\phi(x)}N$ is the parallel transport along the unique length minimizing geodesic between $\phi(x)$ and $\tilde\phi(x)$. Then the map $T_{\phi,\tilde\phi}$ defined by $(T_{\phi,\tilde\phi}\psi)(x)=T_{\phi,\tilde\phi}(x)\psi(x)$ for a.e. $x\in M$ is a bounded linear map $T_{\phi,\tilde\phi}:H^{1/2}(M,\Sigma M\otimes\phi^*TN)\to H^{1/2}(M,\Sigma M\otimes\tilde\phi^*TN)$. Moreover, its operator norm satisfies 
\begin{equation}\label{norm of T}
\|T_{\phi,\tilde\phi}\|_{H^{1/2}\to{H^{1/2}}}\leq C(\|\phi\|_{1,2\alpha},\|\tilde\phi\|_{1,2\alpha})
\end{equation} for some constant $C(\|\phi\|_{1,2\alpha},\|\tilde\phi\|_{1,2\alpha})$ depending only on $\|\phi\|_{1,2\alpha}$ and $\|\tilde\phi\|_{1,2\alpha}$.

Consider $T_{\phi,\tilde\phi}$ as $T_{\phi,\tilde\phi}:H^{1/2}(M,\Sigma M\otimes\phi^*TN)\to H^{1/2}(M,\Sigma M\otimes\underline{\mathbb R}^L)$. For any $\psi\in H^{1/2}(M,\Sigma M\otimes\phi^*TN)\subset H^{1/2}(M,\Sigma M\otimes\underline{\mathbb R}^L)$ we have
\begin{equation}\label{estimate Tpsi}
\|T_{\phi,\tilde\phi}\psi-\psi\|_{1/2,2}\leq C(\|\phi\|_{1,2\alpha},\|\tilde\phi\|_{1,2\alpha})\|\phi-\tilde\phi\|_{1,2\alpha}\|\psi\|_{1/2,2}
\end{equation}
for some constant $C(\|\phi\|_{1,2\alpha},\|\tilde\phi\|_{1,2\alpha})$ depending only on $\|\phi\|_{1,2\alpha}$ and $\|\tilde\phi\|_{1,2\alpha}$.
\end{lem}
\begin{lem}\label{P+}
Let $\phi,\tilde\phi\in  W^{1,2\alpha}(M,N)$. We denote by $P^+(\phi): \mathcal{F}_{\phi}\to\mathcal{F}_{\phi}^+$ the spectral projection onto the positive eigenspace of $\slashed{D}_{\phi}$. For $\psi\in\mathcal{F}_{\phi}:=H^{1/2}(M,\Sigma M\otimes\phi^*TN)$, we have
\begin{equation}
\|P^+(\tilde\phi)\circ T_{\phi,\tilde\phi}\circ P^+(\phi)\psi-P^+(\phi)\psi\|_{1/2,2}\leq C(\|\phi\|_{1,2\alpha},\|\tilde\phi\|_{1,2\alpha})\|\phi-\tilde\phi\|_{1,2\alpha}\|\psi\|_{1/2,2},
\end{equation}
where the $H^{1/2}$-norm on the left side is taken in $H^{1/2}(M,\Sigma M\otimes\underline{\mathbb R}^L)$, $T_{\phi,\tilde\phi}$ is defined as in Lemma \ref{T} and $C(\|\phi\|_{1,2},\|\tilde\phi\|_{1,2\alpha})$ depending only on $\|\phi\|_{1,2\alpha}$ and $\|\tilde\phi\|_{1,2\alpha}$.
\end{lem}

\begin{lem}\label{tildeD}
Let $\phi:[0,1]\ni t\mapsto\phi_t\in W^{1,2\alpha}(M,N)$ be piecewise $C^1$ and ${\bf P}_{t}(x):T_{\phi_0(x)}N\to T_{\phi_t(x)}N$ the parallel transport along the curve $[0,1]\ni s\mapsto\phi_s(x)\in N$. Set $\tilde{\slashed D}_t:={\bf P}_{t}^{-1}\circ\slashed{D}_{\phi_t}\circ{\bf P}_{t}$. Then $K_t:=({\bf1}+{\tilde{\slashed D}_t}^2)^{-1/2}\tilde{\slashed D}_t-({\bf1}+{\slashed D}_{\phi_0}^2)^{-1/2}{\slashed D}_{\phi_0}:H^{1/2}(M,\Sigma M\otimes\phi_0^*TN)\to H^{1/2}(M,\Sigma M\otimes\phi_0^*TN)$ defines a continous family of compact operators with respect to the operator norm.
\end{lem}

Note that the decomposition of $\nabla^V\mathcal{L}^\alpha$ is not unique. For our purposes, we need to use a decomposition that is different from those presented in \eqref{decomposition1}. Let $T>0$ be arbitrary. For $t\in[0,T]$, we write
\begin{equation}
\nabla^V\mathcal{L}^\alpha(\phi_t,\psi_t)={\bf P}^{-1}_{t\to T}\circ({\bf 1}+\slashed{D}^2_{\phi_T})^{-1/2}\circ\slashed{D}_{\phi_T}\circ{\bf P}_{t\to T}\psi_t+K(T,t;\phi_t,\psi_t),
\end{equation}
where $\slashed{D}_{\phi_T}$ denotes the Dirac operator along the map $\phi_T$, ${\bf P}^{-1}_{t\to T}={\bf1}\otimes{\bf P}^{-1}_{t\to T}$ and ${\bf P}^{-1}_{t\to T}(x): T_{\phi_t(x)}N\to T_{\phi_T(x)}N$ is the parallel transport along the curve $[t,T]\ni s\mapsto\phi_s(x)\in N$. Then \eqref{modifiedpgf} can be written as 
\begin{equation}\label{modifiedpgf'}
\frac{d}{dt}\tilde\psi_t=-\eta_t(\tilde\psi_t)(L_{T,V}\tilde\psi_t+\tilde{K}(T,t;\tilde\psi_t)),
\end{equation}
where $L_{T,V}:={\bf P}_T^{-1}\circ({\bf1}+{\slashed D}_{\phi_T}^2)^{-1/2}{\slashed D}_{\phi_T}\circ{\bf P}_T$ and $\tilde{K}(T,t;\tilde\psi_t)={\bf P}_t^{-1}K(T,t;\phi_t,\psi_t)$. As in Lemma \ref{tildeD}, ${\bf P}_tL_{T,V}-L^V{\bf P}_t:H^{1/2}(M,\Sigma M\otimes\phi^*TN)\to H^{1/2}(M,\Sigma M\otimes\underline{\mathbb R}^L)$ is compact and continuously depends on $t$ and $T$. So, $\tilde K$ is also compact.
 
 Regarding $\phi_t$ and $\tilde\psi_t$ in $\eta_t(\tilde\psi_t)$ and $\tilde{K}(T,t;\tilde{\psi}_t)$ as already known functions, integration of of \eqref{modifiedpgf'} yields 
 \begin{equation}
 \tilde\psi_t=\exp\Bigg(-\int_0^t\eta_s(\tilde\psi_s)dsL_{T,V}\Bigg)\psi+\hat{K}(T,t;\psi),
 \end{equation}
 where 
 \begin{equation}
 \hat{K}(T,t;\psi)=-\int_0^t\exp\Bigg(\int_t^\tau\eta_s(\tilde\psi_s)dsL_{T,V}\Bigg)\eta_\tau(\tilde\psi_\tau)\tilde{K}(T,\tau;\tilde\psi_\tau)d\tau
 \end{equation} is a compact map from $[0,1]\times[0,1]\times H^{1/2}(M,\Sigma M\otimes\phi^*TN)$ to $H^{1/2}(M,\Sigma M\otimes\phi^*TN)$ (see Lemma 7.3 in \cite{isobe2012existence}).
 
Since $t\in[0,T]$ is arbitrary, we take $T=t$ and obtain
\begin{equation}\label{finalexpression}
\tilde\psi_t=\exp\Bigg(-\int_0^t\eta_s(\tilde\psi_s)dsL_{t,V}\Bigg)\psi+K(t,\psi),
 \end{equation}
where $K(t,\psi):=\hat{K}(t,t;\psi):[0,1]\times H^{1/2}(M,\Sigma M\otimes\phi^*TN)\to H^{1/2}(M,\Sigma M\otimes\phi^*TN)$ is a compact map.
\begin{rmk}
The advantage of the form of $\tilde\psi_t$ in \eqref{finalexpression} is that the bounded linear operator $\exp\Bigg(-\int_0^t\eta_s(\tilde\psi_s)dsL_{t,V}\Bigg)$ commutes with the spectral projection of the operator ${\bf P}_t^{-1}\slashed{D}_{\phi_t}{\bf P}_t$. This fact will be used in the section below.
\end{rmk}

\subsection{Linking geometry}
In this section, we give a  linking argument parallel to that of \cite{isobe2012existence} within our framework. Since the proofs only rely on those properties generalized in the section above, we again refer to \cite{isobe2012existence} for details.
\begin{defn}\label{linkset}
We define
$$\mathcal{C}:=\{\Theta\in C^0([0,1]\times\mathcal{F},\mathcal{F}):\Theta \ \text{satisfies the  following} \  (1),(2),(3)\}$$
$(1)$ \ $\Theta(0,(\phi,\psi))=(\phi,\psi)$ for $(\phi,\psi)\in\mathcal{F}$.\\
$(2)$ \ For $0\leq t\leq1$ and $(\phi,\psi)\in\mathcal{F}$, writing $\Theta(t,(\phi,\psi))=:(\phi_,\psi_t)$, we have $\phi_t\in W^{1,2\alpha}(M,N)$ and $\psi_t\in H^{1/2}(M,\Sigma M\otimes\phi_t^*TN)$. Moreover, $\phi:[0,1]\ni t\mapsto\phi_t\in W^{1,2\alpha}(M,N)$ is piecewise $C^1$.\\
$(3)$ \ Let $(\phi_t,\psi_t)=\Theta(t,(\phi,\psi))$ be as in $(2)$. For all $0\leq t\leq1$, there holds
$$(P_t^-\oplus P_t^0)\psi_t=P_t^-({\bf P}_t\psi)+P_t^0({\bf P}_t\psi)+K(t,\psi),$$ where ${\bf P}_t\psi(x)=({\bf 1}\otimes{\bf P}_t(x)\psi(x))$ and ${\bf P}_t(x):T_{\phi(x)}N\to T_{\phi_t(x)}N$ is the parallel transport along the curve $[0,1]\ni s\mapsto\phi_s(x)\in N$, $P_t^-$ and $P_t^0$ are the spectral projection of $H^{1/2}(M,\Sigma M\otimes\phi_t^*TN)$ onto the negative and the null eigenspaces of $\slashed{D}_{\phi_t}$, respectively. $K:[0,1]\times H^{1/2}(M,\Sigma M\otimes\phi^*TN)\to ^{1/2}(M,\Sigma M\otimes\underline{\mathbb R}^L)$ is a compact map.
\end{defn}
\begin{defn}
Let $S\subset\mathcal{F}$ be a closed subset and $Q\subset\mathcal{F}$ a submanifold with relative boundary $\pt Q$. We say that $S$ and $Q$ link with respect to $\mathcal C$ if the following holds: 
For any $\Theta\in\mathcal C$ which satisfies $\Theta(t,\pt Q)\cap S=\emptyset$ for all $t\in[0,1]$, we have $\Theta(t,Q)\cap S\neq\emptyset$ for all $t\in[0,1]$.
\end{defn}
\begin{lem}\label{QSlink}
For $\phi_0\in W^{1,2\alpha}(M,N)$ and $R_1,R_2,\rho>0$ with $0<\rho<R_2$, we define
$$S_\rho=\{(\phi,\psi)\in\mathcal{F}^+: \|\psi\|_{1/2,2}=\rho\}$$
and $$Q_{R_1,R_2}=\{(\phi_0,\psi)\in\mathcal{F}_{\phi_0}^-\oplus\mathcal{F}_{\phi_0}^0:\|\psi\|_{1/2,2}\leq R_1\}\oplus\{(\phi_0,re^+):0\leq r\leq R_2\},$$ where $\mathcal{F}=\mathcal{F}^-\oplus\mathcal{F}^0\oplus\mathcal{F}^+$ is the spectral decomposition of $\mathcal{F}$ with respect to the operator ${\bf D}$. The fiber over $\phi\in W^{1,2\alpha}(M,N)$, $\mathcal{F}_\phi:=H^{1/2}(M,\Sigma M\otimes\phi^*TN)$ is decomposed as $\mathcal{F}_\phi=\mathcal{F}_\phi^-\oplus\mathcal{F}_\phi^0\oplus\mathcal{F}_\phi^+$, where $\mathcal{F}_\phi^-$, $\mathcal{F}_\phi^0$ and $\mathcal{F}_\phi^+$ are, respectively, the positive, the null, and the negative spaces with respect to the spectral decomposition of the operator $\slashed{D}_\phi$ and $e^+\in\mathcal{F}_{\phi_0}^+$ is such that $\|e^+\|_{1/2,2}=1$. Then $S_\rho$ and $Q_{R_1,R_2}$ link with respect to $\mathcal C$.
\end{lem}
 
 Now, we give an explicit deformation class, which is useful in this paper.
 \begin{defn}\label{Gamma}
 We define
 $$\Gamma(Q):=\{\gamma(1,\cdot)\in C^0(\mathcal{F},\mathcal{F}):\gamma\in C^0([0,1]\times\mathcal{F},\mathcal{F}) \ \text{satisfies the  following} \  (1),(2),(3)\}$$
 $(1)$ \ $\gamma(0,(\phi,\psi))=(\phi,\psi)$ for any$(\phi,\psi)\in\mathcal{F}$.\\
 $(2)$ \ $\gamma(t,(\phi,\psi))=(\phi,\psi)$ for $0\leq t\leq1$ and $(\phi,\psi)\in\pt Q$.\\
 $(3)$ \ $\gamma(t,(\phi,\psi))$ is written as $\gamma(t,(\phi,\psi))=(\phi_t,\psi_t)$, where $\phi_t\in W^{1,2\alpha}(M,N)$, $\psi\in H^{1/2}(M,\Sigma M\otimes\phi_t^*TN)$ and $\phi:[0,1]\ni t\mapsto\phi_t\in W^{1,2\alpha}(M,N)$ is piecewise $C^1$. Moreover, $\psi_t$ is of the following form:
 $$\psi_t={\bf P}_t(L_{\phi_t}(\psi)+K(t,\psi)),$$
 where ${\bf P}_t$ is as in Definition \ref{linkset}, $K:[0,1]\times H^{1/2}(M,\Sigma M\otimes\phi^*TN)\to H^{1/2}(M,\Sigma M\otimes\phi^*TN)$ is a compact map and $L_{\phi_t}(\psi)=\exp(-\sigma(t;\phi,\psi)L_{t,V})\psi$. Here $L_{t,V}={\bf P}_t^{-1}\circ({\bf1}+\slashed{D}_{\phi_t}^2)^{-1/2}\slashed{D}_{\phi_t}\circ{\bf P}_t$ and $\sigma:[0,1]\times\mathcal{F}\to\mathbb{R}$ is a continuous bounded function.
 \end{defn}
 \begin{lem}\label{QSintersect}
 Let $S\subset\mathcal{F}$ be a closed subset and $Q$ a submanifold of $\mathcal F$ with the relative boundary $\pt Q$. Assume that $S\subset\mathcal{F}^+$ and $S$ and $Q$ link with respect to $\mathcal C$. Then for any $\gamma\in\Gamma(Q)$, we have $\gamma(Q)\cap S\neq\emptyset$.
 \end{lem}

\section{Uniqueness of $\alpha$-harmonic maps}
The uniqueness of harmonic maps can be found in \cite{jost2017riemannian}(see Theorem 9.7.2). In this section, for completeness, we prove the analogous theorem for $\alpha$-harmonic maps. The key is still the convexity of the $\alpha$-energy function. So we now derive the second variation formula for the $\alpha$-energy.

Let $f_{st}(x):=f(x,s,t)$ be a smooth family of maps from $M\times(-\varepsilon, \varepsilon)\times(-\varepsilon, \varepsilon)$ to $N$. The $\alpha$-energy function is 
\begin{equation}
E^{\alpha}(f)=\frac{1}{2}\int_M(1+|df|^2)^{\alpha},
\end{equation}
where $df=\frac{\partial f}{\partial x^\alpha}dx^\alpha=\frac{\partial f^i}{\partial x^\alpha}dx^\alpha\otimes\frac{\partial}{\partial f^i}$. In the following, $\nabla$ is the Levi-Civita connection in the pullback bundle $f^*TN$ and everything will be evaluated at $s=t=0$. Then
\begin{equation}\label{2ndvar1}
\begin{split}
&\left.\frac{\partial^2}{\partial t\partial s}\right|_{s=t=0}\left(\frac12(1+|df|^2)^{\alpha}\right)\\
&=\frac{\partial}{\partial t}\left(\alpha(1+|df|^2)^{(\alpha-1)}\langle \nabla_{\frac{\partial}{\partial s}}df, df\rangle  \right)\\
&=\frac{\partial}{\partial t}\left(\alpha(1+|df|^2)^{(\alpha-1)}\langle \nabla_{\frac{\partial}{\partial x^\beta}}\frac{\partial f}{\partial s}dx^\beta, df\rangle  \right)\\
&=\alpha(1+|df|^2)^{(\alpha-1)}\left(\langle \nabla_{\frac{\partial}{\partial t}}\nabla_{\frac{\partial}{\partial x^\beta}}\frac{\partial f}{\partial s}dx^\beta, df\rangle+\langle \nabla_{\frac{\partial}{\partial x^\beta}}\frac{\partial f}{\partial s}dx^\beta, \nabla_{\frac{\partial}{\partial x^\gamma}}\frac{\partial f}{\partial t}dx^\gamma\rangle\right)\\
&+2\alpha(\alpha-1)(1+|df|^2)^{(\alpha-2)}\langle \nabla_{\frac{\partial}{\partial x^\beta}}\frac{\partial f}{\partial t}dx^\beta, df\rangle\langle \nabla_{\frac{\partial}{\partial x^\gamma}}\frac{\partial f}{\partial s}dx^\gamma, df\rangle.
\end{split}
\end{equation}
Let us set $V:=\left.\frac{\partial f}{\partial s}\right|_{s=t=0}$ and $W:=\left.\frac{\partial f}{\partial t}\right|_{s=t=0}$. Then \eqref{2ndvar1} becomes
\begin{equation}\label{2ndvar2}
\begin{split}
&\left. \frac{\partial^2}{\partial t\partial s}\right|_{s=t=0}\left(\frac12(1+|df|^2)^{\alpha}\right)\\
&=\alpha(1+|df|^2)^{(\alpha-1)}\left( \langle \nabla_{\frac{\partial}{\partial t}}\nabla_{\frac{\partial}{\partial x^\beta}}\frac{\partial f}{\partial s}dx^\beta, df\rangle+\langle \nabla V, \nabla W\rangle\right)\\
&+2\alpha(\alpha-1)(1+|df|^2)^{(\alpha-2)}\langle\nabla_{\frac{\partial}{\partial x^\beta}}Vdx^\beta, df\rangle\langle\nabla_{\frac{\partial}{\partial x^\gamma}} Wdx^\gamma,df\rangle.
\end{split}
\end{equation}
Here, by the Ricci identity, we have
\begin{equation}\label{ricciid}
\begin{split}
&\langle \nabla_{\frac{\partial}{\partial t}}\nabla_{\frac{\partial}{\partial x^\beta}}\frac{\partial f}{\partial s}dx^\beta, df\rangle\\
&=\langle \nabla_{\frac{\partial}{\partial x^\beta}}\nabla_{\frac{\partial}{\partial t}}\frac{\partial f}{\partial s}dx^\beta+R^N(\frac{\partial f}{\partial t}, \frac{\partial f}{\partial x^\beta})\frac{\partial f}{\partial s}dx^\beta, df\rangle\\
&=\langle\nabla_{\frac{\partial}{\partial x^\beta}}\nabla_{\frac{\partial}{\partial t}}\frac{\partial f}{\partial s}dx^\beta, df\rangle-{\rm tr}_{M}\langle R^N(df, V)W, df\rangle_{f^*TN},
\end{split}
\end{equation}
where we used the definition of the Riemann curvature tensor $R(X,Y)Z=\nabla_X\nabla_YZ-\nabla_Y\nabla_XZ-\nabla_{[X,Y]}Z$. Plugging \eqref{ricciid} into \eqref{2ndvar2}, we obtain the  following second variation formula.
\begin{thm}\label{thm:2ndvar}
For a smooth family $f=f_{st}$ of finite $\alpha$-energy maps between manifolds $M$ and $N$ with $f_{st}(x)=f_{00}(x)$ for any point $x\in\partial M$ in the case of $\partial M\neq\emptyset$, the second variational formula with $V:=\left.\frac{\partial f}{\partial s}\right|_{s=t=0}, W:=\left.\frac{\partial f}{\partial t}\right|_{s=t=0}$ is 

\begin{equation}\label{2ndvar}
\begin{split}
&\left.\frac{\partial^2}{\partial t\partial s}\right|_{s=t=0}E^{\alpha}(f)\\
&=\alpha\int_M(1+|df|^2 )^{(\alpha-1)}(\langle \nabla V, \nabla W\rangle-{\rm tr}_{M}\langle R^N(df, V)W, df\rangle_{f^*TN})\\
&+2\alpha(\alpha-1)\int(1+|df|^2)^{(\alpha-2)}\langle\nabla_{\frac{\partial}{\partial x^\beta}}Vdx^\beta, df\rangle\langle\nabla_{\frac{\partial}{\partial x^\gamma}} Wdx^\gamma,df\rangle\\
&+\alpha\int_M(1+|df|^2)^{(\alpha-1)}\langle\nabla_{\frac{\partial}{\partial x^\beta}}\nabla_{\frac{\partial}{\partial t}}\frac{\partial f}{\partial s}dx^\beta, df\rangle.
\end{split}
\end{equation}
\end{thm}

Since the Euler-Lagrangian equation for a $\alpha$-harmonic map is 
\begin{equation}\label{EL}
\tau^\alpha(f)={\rm tr}\nabla((1+|df|^2)^{(\alpha-1)}df)=0,
\end{equation}
under the assumptions of Theorem \ref{thm:2ndvar}, we can deal with the last term in \eqref{2ndvar} as follows
\begin{equation}\label{lastterm}
\begin{split}
&\alpha\int_M(1+|df|^2 )^{(\alpha-1)}\langle\nabla_{\frac{\partial}{\partial x^\beta}}\nabla_{\frac{\partial}{\partial t}}\frac{\partial f}{\partial s}dx^\beta, df\rangle\\
&=-\alpha\int_M\langle\nabla_{\frac{\partial}{\partial t}}\frac{\partial f}{\partial s}dx^\beta, \nabla_{\frac{\partial}{\partial x^\beta}}((1+|df|^2 )^{(\alpha-1)}df)\rangle\\
&=-\alpha\int_M\langle\nabla_{\frac{\partial}{\partial t}}\frac{\partial f}{\partial s}, {\rm tr}\nabla((1+|df|^2)^{(\alpha-1)}df)\rangle.
\end{split}
\end{equation}
Therefore, plugging \eqref{lastterm} into \eqref{2ndvar}, we have 
\begin{cor}
All  $\alpha$-harmonic maps to non-positive curved manifolds are stable.
\end{cor}

When the target manifold has non-positive curvature, any two homotopic maps can be connected by geodesics. In this case, we also have the convexity of the $\alpha$-energy.

\begin{lem}\label{lem:convexity}
Under the assumptions of Theorem \ref{thm:2ndvar} with $f=f_t$, for $\alpha\geq 1$, if $N$ has non-positive sectional curvature and
\begin{equation}
\left.\nabla_{\frac{\partial}{\partial t}}\frac{\partial f}{\partial t}\right|_{t=0}\equiv0,
\end{equation}
then the $\alpha$-energy function is convex, that is,
$$
\left.\frac{d^2}{dt^2}\right|_{t=0}E^\alpha(f_t)\geq\alpha\int_M(1+|df|^2 )^{(\alpha-2)}\left((1+|df|^2)|\nabla V|^2+2(\alpha-1)\langle\nabla_{\frac{\partial}{\partial x^\beta}}Vdx^\beta, df\rangle^2\right)\geq0.
$$
\end{lem}

By such convexity, one can prove the  uniqueness of $\alpha$-harmonic map in the classical way (see \cite{jost2017riemannian} Theorem 9.7.2).
\begin{thm}\label{thm:uniqueness}
Let $M$ be a compact manifold and $N$ a complete Rienmannian manifold. Assume that $N$ has non-positive curvature. Let $f_0, f_1: M\to N$ be homotopic $\alpha$-harmonic maps ($\alpha\geq1$). Then there exists a family $f_t: M\to N$ ($t\in[0,1]$) of $\alpha$-harmonic maps connecting them, for which the $\alpha$-energy $E^\alpha(f_t)$ is independent of $t$, and for which every curve $\gamma_x(t):=f_t(x)$ is geodesic, and $\|\frac{\partial}{\partial t}\gamma_x(t)\|$ is independent of $x$ and $t$. If $N$ has negative curvature, then $f_0$ and $f_1$ either are both constant maps, or they both map $M$ onto the same closed geodesic or they coincide.

 If $M$ has boundary $\partial M$ and $f_0|_{\partial M}=f_1|_{\partial M}$, then $f_0=f_1$.
\end{thm}

\section{Existence results}
In this section, we will prove the main existence results for perturbed $\alpha$-Dirac-harmonic map. To do so, we still need some more preparations.

Let $N$ be a compact Riemannian manifold and $\theta\in[M,N]$ be a free homotopy class of maps in $N$. We define $\mathcal{F}_\theta=\{(\phi,\psi)\in\mathcal{F}^{\alpha,1/2}(M,N):\phi\in\theta\}$. For each $\alpha$, it is well-known that there exists an $\alpha$-energy minimizing map $\phi_0\in\theta$ in the class $\theta$ \cite{sacks1981existence}. From now on, we fix $\alpha>1$. Denote 
\begin{equation}
m_\theta:=\inf\left\{\frac12\int_M(1+|d\phi|^2)^\alpha:\phi\in W^{1,2\alpha}(M,N)\cap\theta\right\}=\frac12\int_M(1+|d\phi_0|^2)^\alpha.
\end{equation}
Now, we fix such a map $\phi_0\in\theta$ in the following. For $R_1,R_2>0$ and $0<\rho<R_2$, we define
\begin{equation}\label{QR1R2}
Q_{\theta;R_1,R_2}=\{(\phi_0,\psi)\in\mathcal{F}^-_{\theta,\phi_0}\oplus\mathcal{F}^0_{\theta,\phi_0}:\|\psi\|_{1/2,2}\leq R_1\}\oplus\{(\phi_0,re^+):0\leq r\leq R_2\}
\end{equation}
and
\begin{equation}
S_{\theta;\rho}=\{(\phi,\psi)\in\mathcal{F}^+_{\theta}:\|\psi\|_{1/2,2}=\rho\},
\end{equation}
where $e^+\in\mathcal{F}^+_{\theta,\phi_0}$ is such that $\|e^+\|_{1/2,2}=1$ and $\mathcal{F}_{\theta,\phi}=\mathcal{F}^-_{\phi}=H^{1/2}(M,\Sigma M\otimes\phi^*TN)$ is the fiber over $\phi$ of the fibration $\mathcal{F}_{\theta}\to W^{1,2\alpha}(M,N)\cap\theta$. $\mathcal{F}_{\theta}^\pm,\mathcal{F}_{\theta}^0$ are similar to the ones in Lemma \ref{QSlink}. We also define
\begin{equation}
a_{\theta;R_1,R_2}=\sup\{\mathcal{L}^\alpha(\phi,\psi):(\phi,\psi)\in\pt Q_{\theta;R_1,R_2}\}
\end{equation}
and 
\begin{equation}
b_{\theta;\rho}=\inf\{\mathcal{L}^\alpha(\phi,\psi):(\phi,\psi)\in S_{\theta;\rho}\}.
\end{equation}

Before we prove the existence of a critical point of $\mathcal{L}^\alpha$, we need some crucial estimates for $a_{\theta;R_1,R_2}$ and $b_{\theta;\rho}$.

\subsection{Estimate of $a_{\theta;R_1,R_2}$}
In this subsection, we prove the following estimate.
\begin{lem}\label{estimateofa}
There exists $R_1,R_2>0$ such that $a_{\theta;R_1,R_2}\leq m_\theta$ provided $F$ satisfies $(\rm F2)$ and

$(\rm F4)$ \ $F(\phi,\psi)\geq0$ for any $(\phi,\psi)\in\mathcal{F}$.

\end{lem}
\begin{proof}
Write $\psi=\psi^-+\psi^0+re^+$, we have
\begin{equation}
\mathcal{L}^\alpha(\phi_0,\psi)=m_\theta+\frac12\int_M\langle\psi^-,\slashed{D}_{\phi_0}\psi\rangle+C(e^+)r^2-\int_MF(\phi_0,\psi),
\end{equation}
where $C(e^+)=\frac12\int_M\langle\slashed{D}_{\phi_0}e^+,e^+\rangle>0$.

By (F2), we get
\begin{equation}
\int_MF(\phi_0,\psi)\geq C\int_M|\psi|^\mu-C.
\end{equation}

On the other hand, H\"older's inequality implies
\begin{equation}
\int_M|\psi^-+\psi^0|^2+r^2\int_M|e^+|^2=\int_M|\psi^-+\psi^0+re^+|^2\leq C\left(\int_M|\psi|^\mu\right)^{2/\mu}.
\end{equation}
Combining these, we obtain
\begin{equation}\label{1stestimate}
\begin{split}
\mathcal{L}^\alpha(\phi_0,\psi)\leq&m_\theta+\frac12\int_M\langle\psi^-,\slashed{D}_{\phi_0}\psi\rangle+C(e^+)r^2\\
&-C\left(\int_M|\psi^-+\psi^0|^2+r^2\int_M|e^+|^2\right)^{\mu/2}+C.
\end{split}\end{equation}
Since the square of the $H^{1/2}$-norm of $\psi^-\in\mathcal{F}_{\phi_0}^-$ is equivalent to $-\int_M\langle\psi^-,\slashed{D}_{\phi_0}\psi^-\rangle$, there exists $C(\phi_0)>0$ such that
\begin{equation}\label{normpsi}
\int_M\langle\psi^-,\slashed{D}_{\phi_0}\psi^-\rangle\leq-C(\phi_0)\|\psi^-\|^2_{1/2,2}.
\end{equation}
Plugging \eqref{normpsi} into \eqref{1stestimate} and noting that the $L^2$-norm is equivalent to the $H^{1/2}$-norm on $\mathcal{F}^0_{\theta,\phi_0}$, one gets
\begin{equation}\label{finalestimate}
\mathcal{L}^\alpha(\phi_0,\psi)\leq m_\theta+C(e^+)r^2-C(\phi_0)\|\psi^-\|^2_{1/2,2}-C(\|\psi^0\|^\mu_{1/2,2}+r^\mu)+C.
\end{equation}
We choose $R_2>0$ such that $C(e^+)r^2-Cr^\mu+C\leq0$ for $r\geq R_2$. This is possible because $\mu>2$. We then set $M_1=\max\limits_{0\leq r\leq R_2}(C(e^+)r^2-Cr^\mu+C)>0$ and take $R_1$ such that $C(\phi_0)\|\psi^-\|^2_{1/2,2}+C\|\psi^0\|^\mu_{1/2,2}\geq M_1$ whenever $\|\psi^-+\psi^0\|^2_{1/2,2}=\|\psi^-\|^2_{1/2,2}+\|\psi^0\|^2_{1/2,2}\geq R_1^2$.

For such $R_1$ and $R_2$, the lemma holds. Indeed, any point in $\pt Q_{R_1,R_2}$ is in one of the following three subsets:

(1) \ $\|\psi^-+\psi^0\|_{1/2,2}=R_1$ and $0\leq r\leq R_2$;

(2) \ $\|\psi^-+\psi^0\|_{1/2,2}\leq R_1$ and $r=0$;

(3) \ $\|\psi^-+\psi^0\|_{1/2,2}\leq R_1$ and $r=R_2$.

For (1) and (3), by \eqref{finalestimate}, the choices of $R_1$ and $R_2$ imply $\mathcal{L}^\alpha(\phi_0,\psi)\leq m_\theta$.

For (2), by (F4), we have
\begin{equation}
\begin{split}
\mathcal{L}^\alpha(\phi_0,\psi)&=m_\theta+\frac12\int_M\langle\psi^-,\slashed{D}_{\phi_0}\psi\rangle-\int_MF(\phi_0,\psi)\\
&\leq m_\theta-\int_MF(\phi_0,\psi)\leq m_\theta.
\end{split}\end{equation}
This case analysis completes the proof.

\end{proof}

\subsection{Estimate of $b_{\theta,\rho}$}
In this subsection, we will prove 
\begin{lem}\label{estimateofb}
For $R_2$ in Lemma \ref{estimateofa}, there exists $0<\rho<R_2$ such that 
\begin{equation}\label{estiamte of b}
b_{\theta,\rho}>m_\theta
\end{equation} 
provided $F$ satisfies $(\rm F1),(\rm F4)$ and 

$(\rm F5)$ \ $F(\phi,\psi)=o(|\psi|^2)$ uniformly in $\phi\in N$ as $|\psi|\to0$.

 \end{lem}

We need some preparations before we prove the lemma above. For $(\phi,\psi)\in S_{\theta,\rho}$, we have
\begin{equation}
\mathcal{L}^{\alpha}(\phi,\psi)=\frac12\int_{M}(1+|d\phi|^2)^{\alpha}+\frac12\int_{M}\langle\psi,\slashed{D}\psi\rangle_{\Sigma M\otimes\phi^*TN}-\int_{M}F(\phi,\psi).
\end{equation}
We first estimate the term $\frac12\int_{M}\langle\psi,\slashed{D}\psi\rangle_{\Sigma M\otimes\phi^*TN}$. For this, we define
\begin{equation}
\lambda^+(\phi)=\inf\left\{\int_{M}\langle\psi,\slashed{D}\psi\rangle:\|\psi\|_{1/2,2}=1, \ \psi\in\mathcal{F}^+_\phi\right\}.
\end{equation}
Thus, 
\begin{equation}
\int_{M}\langle\psi,\slashed{D}\psi\rangle\geq\lambda^+(\phi)\rho^2
\end{equation}
for $(\phi,\psi)\in S_{\theta,\rho}$.

First of all, we investigate some properties of $\lambda^+(\phi)$.
\begin{lem}\label{positivity}
For any $\phi\in\theta\cap W^{1,2\alpha}(M,N)$, we have $\lambda^+(\phi)>0$.
\end{lem}
\begin{proof}
By definition of $\lambda^+(\phi)$, $\lambda^+(\phi)\geq0$ for any $\phi\in\theta\cap W^{1,2\alpha}(M,N)$. To prove the lemma, we assume on the contrary that $\lambda^+(\phi)=0$ for some $\phi\in\theta\cap W^{1,2\alpha}(M,N)$. In other words, there exist $\psi_n\in\mathcal{F}^+_{\phi}$ such that 
\begin{equation}\label{psin}
\|\psi_n\|_{1/2,2}=1 \ \text{and} \ \int_{M}\langle\psi_n,\slashed{D}\psi_n\rangle\to0
\end{equation}
as $n\to\infty$.
 
 We assume that (taking a subsequence if necessary) there exists $\psi_\infty\in\mathcal{F}^+_\phi$ such that $\psi_n\rightharpoonup\psi_\infty$ weakly in $H^{1/2}$ and $\psi_n\rightharpoonup\psi_\infty$ strongly in $L^2$ as $n\to\infty$. 
 
 As before, the square root of $\psi\mapsto\int_M\langle\psi,\slashed{D}\psi\rangle+\|\psi\|^2_2$ defines a norm on $\mathcal{F}^+_\phi$ which is equivalent to the $H^{1/2}$-norm. By the weak lower semi-continuity of the norm, we get
 \begin{equation}\label{lowersemi}
 \int_M\langle\psi_\infty,\slashed{D}\psi_\infty\rangle+\|\psi_\infty\|^2_2\leq\varliminf\limits_{n\to\infty}\left(\int_M\langle\psi_n,\slashed{D}\psi_n\rangle+\|\psi_n\|^2_2\right).
 \end{equation}
 By \eqref{psin} and \eqref{lowersemi} and $\int_M\langle\psi_\infty,\slashed{D}\psi_\infty\rangle\geq0$ (since $\psi_\infty\in\mathcal{F}^+_\phi$), we obtain $\int_M\langle\psi_\infty,\slashed{D}\psi_\infty\rangle=0$. This implies $\psi_\infty=0$. Then there is a constant $C$ such that
 \begin{equation}
 1=\|\psi_n\|^2_{1/2,2}\leq C\left(\int_M\langle\psi_n,\slashed{D}\psi_n\rangle+\|\psi_n\|^2_2\right)\to0
 \end{equation}
 as $n\to\infty$. This is a contradiction. So we complete the proof.
 
\end{proof}

We next prove:
\begin{lem}\label{continuous}
The map $W^{1,2\alpha}(M,N)\ni\phi\mapsto\lambda^+(\phi)\in\mathbb{R}$ is continuous.
\end{lem}
\begin{proof}
Let $\phi\in W^{1,2\alpha}(M,N)$. By the definition of $\lambda^+(\phi)$, for any $\varepsilon>0$, there exists $\psi_\varepsilon\in\mathcal{F}^+_\phi$ such that 
\begin{equation}\label{psivarepsilon}
\|\psi_\varepsilon\|_{1/2,2}=1 \ \text{and} \ \int_{M}\langle\psi_\varepsilon,\slashed{D}\psi_\varepsilon\rangle<\lambda^+(\phi)+\varepsilon.
\end{equation}
Let $\tilde\phi\in W^{1,2\alpha}\cap L^\infty(M,N)$ be another map such that $\|\phi-\tilde\phi\|_{W^{1,2\alpha}\cap L^\infty}$ is  small. Because $\psi_\varepsilon\notin\mathcal{F}^+_{\tilde\phi}$ in general, we cannot use $\psi_\varepsilon$ as a test spinor to estimate $\lambda^+(\tilde\phi)$. In order to get a suitable test spinor, we need parallel translation. Denote by $\iota(N)>0$ the injectivity radius of $N$. For any $y,z\in N$ with $d(y,z)<\iota(N)$ ($d(y,z)$ is the geodesic distance between $y$ and $z$ in $N$), ${\bf P}_{y,z}$ defined as in Lemma \ref{T} depends smoothly on $y,z$. By Lemma \ref{T}, $T_{\phi,\tilde\phi}\psi_\varepsilon\in\mathcal{F}_{\tilde\phi}$. To make it belong to $\mathcal{F}^+_{\tilde\phi}$, we need to modify it.

We define $\tilde\psi_\varepsilon:=P^+(\tilde\phi)T_{\phi,\tilde\phi}\psi_\varepsilon\in\mathcal{F}^+_{\tilde\phi}$. By Lemma \ref{P+}, we have
\begin{equation}
\|\tilde\psi_\varepsilon-\psi_\varepsilon\|_{1/2,2}\leq C(\|\phi\|_{1,2\alpha},\|\tilde\phi\|_{1,2\alpha})\|\phi-\tilde\phi\|_{1,2\alpha}.
\end{equation}
This and \eqref{psivarepsilon} give us
\begin{equation}
\begin{split}
\lambda^+(\phi)+\varepsilon&\geq\int_{M}\langle\psi_\varepsilon,\slashed{D}_{\phi}\psi_\varepsilon\rangle\\
&\geq\int_{M}\langle\psi_\varepsilon-\tilde\psi_\varepsilon,\slashed{D}_{\phi}\psi_\varepsilon\rangle+\int_{M}\langle\tilde\psi_\varepsilon,\slashed{D}_{\tilde\phi}\tilde\psi_\varepsilon\rangle+\int_{M}\langle\tilde\psi_\varepsilon,\slashed{D}_\phi\psi_\varepsilon-\slashed{D}_{\tilde\phi}\tilde\psi_\varepsilon\rangle\\
&\geq\lambda^+(\tilde\phi)-C(\|\phi\|_{1,2\alpha},\|\tilde\phi\|_{1,2\alpha})\|\phi-\tilde\phi\|_{1,2\alpha},
\end{split}\end{equation}
where $\langle\cdot,\cdot\rangle$ is the metric on $\Sigma M\otimes\underline{\mathbb R}^L$.

Since $\varepsilon>0$ is arbitrary, we obtain
\begin{equation}\label{differenceoflambda}
\lambda^+(\tilde\phi)-\lambda^+(\phi)\leq C(\|\phi\|_{1,2\alpha},\|\tilde\phi\|_{1,2\alpha})\|\phi-\tilde\phi\|_{1,2\alpha}.
\end{equation}
Reversing the roles of $\phi$ and $\tilde\phi$ in \eqref{differenceoflambda}, we arrive at
\begin{equation}
|\lambda^+(\tilde\phi)-\lambda^+(\phi)|\leq C(\|\phi\|_{1,2\alpha},\|\tilde\phi\|_{1,2\alpha})\|\phi-\tilde\phi\|_{1,2\alpha}.
\end{equation}
This completes the proof.

\end{proof}

Let $\mathcal{M}^\alpha(\theta)=\{\phi\in\theta\cap W^{1,2\alpha}(M,N):E^\alpha(\phi)=m_\theta\}$ be the set of energy minimizing maps in the class $\theta$, where $E^\alpha$ is defined in \eqref{alphaenergy}. Because the $\alpha$-energy of $\phi$ satisfies the Palais-Smale condition for $\alpha>1$\cite{urakawa2013calculus}, we know
\begin{lem}\label{compact}
$\mathcal{M}^\alpha(\theta)$ is compact for $\alpha>1$.
\end{lem}

As a corollary of Lemmas \ref{positivity}, \ref{continuous} and \ref{compact}, we have:
\begin{cor}\label{neighborhood}
There exist $\delta(\theta)$ and $\lambda^+(\theta)>0$ such that for any $\phi\in\theta\cap W^{1,2\alpha}(M,N)$ with ${\rm dist}(\phi,\mathcal{M}^\alpha(\theta)):=\inf\{\|\phi-\varphi\|_{W^{1,2\alpha}}:\varphi\in\mathcal{M}^\alpha(\theta)\}<\delta(\theta)$, there holds $\lambda^+(\phi)\geq\lambda^+(\theta)$. 
\end{cor}
\begin{proof}
Suppose the corollary is not true, then there exist $\phi_n\in\theta\cap W^{1,2\alpha}(M,N)$ such that $\lambda^+(\phi_n)\to0$ and ${\rm dist}(\phi_n,\mathcal{M}^\alpha(\theta))\to0$. By the compactness of $\mathcal{M}(\theta)$, after taking a subsequence if necessary, we assume that there is $\phi_\infty\in\mathcal{M}(\theta)$ such that $\phi\to\phi_\infty$ in $W^{1,2\alpha}(M,N)$. By Lemmas \ref{positivity} and \ref{continuous}, we have $\lambda^+(\phi_n)\to\lambda^+(\phi_\infty)>0$, which is a contradiction. 

\end{proof}


We also need to investigate the maps far away from the set $\mathcal{M}^\alpha(\theta)$.

\begin{lem}\label{outside}
Let $\delta(\theta)>0$ be as in Corollary \ref{neighborhood}. There exists $\varepsilon(\theta)$ such that for any $\phi\in\theta\cap W^{1,2\alpha}(M,N)$ with ${\rm dist}(\phi,\mathcal{M}^\alpha(\theta))\geq\delta(\theta)$, we have $\frac12\int_{M}(1+|d\phi|^2)^\alpha\geq m_\theta+\varepsilon(\theta)$.
\end{lem}
\begin{proof}
For any fixed $\alpha>1$, if such a $\varepsilon(\theta)$ does not exist, then there exist $\{\phi_n\}$ such that
\begin{equation}\label{contradiction}
 {\rm dist}(\phi_n,\mathcal{M}^\alpha(\theta))\geq\delta(\theta)
 \end{equation}
and 
\begin{equation}
\frac12\int_{M}(1+|d\phi_n|^2)^\alpha\leq m_\theta+\frac1n.
\end{equation}
Since $m_\theta$ is the minimizing energy of $E^\alpha$, $\{\phi_n\}$ is a minimizing sequence for $E^\alpha$. Therefore, the Palais-Smale condition implies that, after taking a subsequence if necessary, there is a critical point $\phi_\infty\in\theta\cap W^{1,2\alpha}(M,N)$ of $E^\alpha$ such that $\phi_n\to\phi_\infty$ strongly in $W^{1,2\alpha}(M,N)$. In fact, $\phi_\infty\in\mathcal{M}^\alpha(\theta)$. This contradicts \eqref{contradiction}. 

\end{proof}

Now, we can prove Lemma \ref{estimateofb}.
\begin{proof}[Proof of Lemma \ref{estimateofb}]
Let $(\phi,\psi)\in S_{\theta,\rho}$ be arbitrary. 
If $\phi$ satisfies ${\rm dist}(\phi,\mathcal{M}^\alpha(\theta))<\delta(\theta)$, by Corollary \ref{neighborhood}, we have
\begin{equation}
\begin{split}
\frac12\int_{M}(1+|\phi|^2)^{\alpha}+\frac12\int_{M}\langle\psi,\slashed{D}\psi\rangle&\geq m_\theta+\frac{\lambda^+(\phi)}{2}\|\psi\|^2_{1/2,2}\\
&\geq m_\theta+\frac{\lambda^+(\theta)}{2}\rho^2.
\end{split}\end{equation}

If $\phi$ satisfies ${\rm dist}(\phi,\mathcal{M}^\alpha(\theta))\geq\delta(\theta)$, Lemma \ref{outside} implies
\begin{equation}
\frac12\int_{M}(1+|\phi|^2)^{\alpha}+\frac12\int_{M}\langle\psi,\slashed{D}\psi\rangle\geq m_\theta+\varepsilon(\theta).
\end{equation}
Therefore, we have
\begin{equation}\label{map+spinor}
\frac12\int_{M}(1+|\phi|^2)^{\alpha}+\frac12\int_{M}\langle\psi,\slashed{D}\psi\rangle\geq m_\theta+
\min\left\{\frac{\lambda^+(\theta)}{2}\rho^2,\varepsilon(\theta)\right\}.
\end{equation}

On the other hand, by $(\rm F1),(F4),(F5)$, for any $\epsilon>0$ there exist $C_\epsilon$ such that 
\begin{equation}\label{F}
0\leq F(\phi,\psi)\leq\epsilon|\psi|^2+C_\epsilon|\psi|^p
\end{equation}
for any $(\phi,\psi)\in\mathcal{F}$. 

Now take $\epsilon=\frac{\lambda^+(\theta)}{4}$ and write $C=C_\epsilon$. By \eqref{map+spinor}, \eqref{F} and Sobolev embedding, we obtain
\begin{equation}
\begin{split}
\mathcal{L}^\alpha(\phi,\psi)&=\frac12\int_{M}(1+|\phi|^2)^{\alpha}+\frac12\int_{M}\langle\psi,\slashed{D}\psi\rangle-\int_MF(\phi,\psi)\\
&\geq m_\theta+\min\left\{\frac{\lambda^+(\theta)}{2}\rho^2,\varepsilon(\theta)\right\}-\frac{\lambda^+(\theta)}{4}\int_M|\psi|^2-C\int_M|\psi|^p.
\end{split}
\end{equation}
There exists $0<\rho_0<R_2$ such that for $0<\rho<\rho_0$, we have
\begin{equation}
\min\left\{\frac{\lambda^+(\theta)}{2}\rho^2,\varepsilon(\theta)\right\}=\frac{\lambda^+(\theta)}{2}\rho^2
\end{equation}
and 
\begin{equation}
\frac{\lambda^+(\theta)}{4}\rho^2-C\rho^p>0.
\end{equation}
Thus, for any $\rho\in(0,\rho_0)$, we have $\mathcal{L}^\alpha(\phi,\psi)\geq m_\theta+\frac12\rho^2\varepsilon(\theta)$. This implies
\begin{equation}
b_{\theta;\rho}\geq m_\theta+\frac12\rho^2\varepsilon(\theta)>m_\theta.
\end{equation}
Hence, we complete the proof.

\end{proof}

\begin{rmk}
 It follows from the proof above that the estimate \eqref{estiamte of b} is uniform in $k$ if we replace $F$ by $\frac1kF$ instead.
\end{rmk}

\subsection{Critical value of $\mathcal{L}^\alpha$}
We define
\begin{equation}
c_\theta=\inf\limits_{\gamma\in\Gamma(Q_{\theta;R_1,R_2})}\sup\mathcal{L}^\alpha(\gamma(Q_{\theta;R_1,R_2})),
\end{equation}
where $\Gamma(Q_{\theta;R_1,R_2})$ is defined in Definition \ref{Gamma}. Similar to \cite{isobe2012existence}, by Lemma \ref{tildeD}, one can prove that $\Gamma(Q_{\theta;R_1,R_2})$ is closed under composition. With this property, we obtain the following main existence result in this paper:
\begin{thm}\label{existence}
Let $M$ be a closed surface and $N$ a compact manifold. Suppose $F$ satisfies $(\rm F1)-(F5)$ with $\frac{4\alpha}{3\alpha-2}\leq\mu\leq p\leq\frac34\mu+1$ for $\alpha\in(1,2]$. Let $R_1,R_2$ and $\rho$ be as in Lemmas \ref{estimateofa} and \ref{estimateofb}. Then we have $m_\theta<c_\theta<\infty$ and $c_\theta$ is a critical value of $\mathcal{L}^\alpha$ in $\mathcal{F}$.
\end{thm}
\begin{proof}
By Lemmas \ref{QSlink} and \ref{QSintersect}, we know $\gamma(Q_{\theta;R_1,R_2})\cap S_{\theta;\rho}\neq\emptyset$ for any $\gamma\in\Gamma(Q_{\theta;R_1,R_2})$. Therefore, Lemma \ref{estimateofb} implies
\begin{equation}
c_{\theta}\geq\inf\limits_{(\phi,\psi)\in S_{\theta;\rho}}\mathcal{L}^\alpha(\phi,\psi)=b_{\theta;\rho}>m_{\theta}.
\end{equation}
On the other hand, since the identity map belongs to $\Gamma(Q_{\theta;R_1,R_2})$ and $Q_{\theta;R_1,R_2}$ is bounded, by Sobolev embedding, we get
\begin{equation}
c_\theta\leq\sup\mathcal{L}^\alpha(Q_{\theta;R_1,R_2})<\infty.
\end{equation}
It remains to show that $c_\theta$ is a critical value of $\mathcal{F}^\alpha$. Suppose this is not the case, we set $\bar\varepsilon=\frac{b_{\theta;\rho}-m_{\theta}}{2}>0$. By integrating \eqref{pgf}(see Deformation Lemma in \cite{rabinowitz1986minimax}\cite{struwe2008variational}), we can find $0<\varepsilon<\bar\varepsilon$ and $\Phi\in\Gamma(Q_{\theta;R_1,R_2})$ such that 
\begin{equation}
\Phi(1,\mathcal{L}^\alpha_{c_{\theta}+\varepsilon})\subset\mathcal{L}^\alpha_{c_{\theta}-\varepsilon},
\end{equation}
where $\mathcal{L}^\alpha_{c}=\{(\phi,\psi)\in\mathcal{F}_\theta:\mathcal{L}^\alpha(\phi,\psi)<c\}$. Here, note that our negative psudo-gradient flow belongs to $\Gamma(Q_{\theta;R_1,R_2})$ by \eqref{finalexpression}. Also, observing that $a_{\theta;R_1,R_2}\leq m_{\theta}<c_{\theta}-\bar\varepsilon$, we have $\Phi(1,(\phi,\psi))=(\phi,\psi)$ for any $(\phi,\psi)\in\pt Q_{\theta;R_1,R_2}$.

Choose $\gamma\in\Gamma(Q_{\theta;R_1,R_2})$ such that $\mathcal{L}^\alpha(\gamma(1,Q_{\theta;R_1,R_2}))<c_{\theta}+\varepsilon$. Since $\Phi\circ\gamma\in\Gamma(Q_{\theta;R_1,R_2})$ as we said before the theorem, we have the following contradiction
\begin{equation}
c_{\theta}\leq\sup\mathcal{L}^\alpha(\Phi(1,\gamma(1,Q_{\theta;R_1,R_2})))<c_{\theta}-\varepsilon.
\end{equation}
Thus we complete the proof.

\end{proof}

Theorem \ref{existence} gives us a solution to the perturbed $\alpha$-Dirac-harmonic map equations \eqref{elnonlinearalpha1} and \eqref{elnonlinearalpha2}. Let us denote this solution by $(\phi_\theta,\psi_\theta)$. From (F5), we know $F(\phi,0)=0$ and $F_\psi(\phi,0)=0$ for any $\phi\in N$. Therefore, we cannot exclude the possibility that the solution in Theorem \ref{existence} is a trivial solution, that is, $\psi_\theta=0$. If $\psi_\theta=0$, then $F_\phi(\phi_\theta,0)=0$ by (F5), which tells us that $\phi_\theta$ is an $\alpha$-harmonic map. On the other hand, $\mathcal{L}^\alpha(\phi_\theta,\psi_\theta)=c_\theta>m_\theta$. Therefore, the  $\alpha$-energy of $\phi_\theta$ is $\frac12\int_M(1+|d\phi_\theta|^2)^\alpha=c_\theta>m_\theta$. So, we get a corollary:
\begin{cor}
Under the assumption of Theorem \ref{existence}, we get a critical point $(\phi_\theta,\psi_\theta)$ with value $c_\theta$. Then one of the following holds:

$(1)$ \ There exists a non-trivial perturbed $\alpha$-Dirac-harmonic map $(\phi_\theta,\psi_\theta)$ on $N$; or

$(2)$ \ There exists an energy non-minimizing $\alpha$-harmonic map $\phi_\theta$ in $\theta$. 
\end{cor}

In particular, when $N$ has non-positive curvature, Theorem \ref{thm:uniqueness} excludes the possibility (2) in the corollary above. Finally, we obtain an existence result about non-trivial solutions.
\begin{thm}\label{nontrivial}
Let $M$ be a closed surface and $N$ a compact Rienmannian manifold with non-positive curvature.
Suppose $F$ satisfies $(\rm F1)-(F5)$ with $\frac{4\alpha}{3\alpha-2}\leq\mu\leq p\leq\frac34\mu+1$ for $\alpha\in(1,2]$. Then for any homotopy class $\theta\in[M,N]$, there exists a non-trivial solution $(\phi,\psi)\in W^{1,2\alpha}(M,N)\times H^{1/2}(M,\Sigma M\otimes\phi^*TN)$ to the perturbed $\alpha$-Dirac-harmonic map equations \eqref{elnonlinearalpha1} and \eqref{elnonlinearalpha2} with $\phi\in\theta$.
\end{thm}

\section{Regularity theorem}
In this section, assuming $1<\alpha\leq2$ as before, we are going to prove a regularity theorem for a weakly perturbed $\alpha$-Dirac-harmonic map, defined as a critical point $(\phi,\psi)\in\mathcal{F}^{\alpha,1/2}:=W^{1,2\alpha}(M,N)\times H^{1/2}(M,\Sigma M\otimes\phi^*TN)$. Precisely, we want to prove:
\begin{thm}\label{regularity}
Suppose $F\in C^\infty$ satisfies $(\rm F1)$ and $(\rm F3)$ for some $p\leq2+2/\alpha$ and $q\geq0$. Then any weakly perturbed $\alpha$-Dirac-harmonic map is smooth. 
\end{thm}

Since the smoothness is a local property, it suffices to prove the local version of Theorem \ref{regularity}. So we fix a domain $\Omega\subset M$, which is mapped into a local chart $\{y^i\}_{i=1,\dots,n}$ on $N$. Then the weakly perturbed $\alpha$-Dirac-harmonic map $(\phi,\psi)$ satisfies the following system in the  weak sense.
\begin{equation}\label{elnonlinearalphalocal1}
\begin{split}
{\rm div}((1+|d\phi|^2)^{\alpha-1}\nabla\phi^m)=&-\Gamma^m_{ij}\phi^i_\beta\phi^j_{\gamma}g^{\beta\gamma}(1+|d\phi|^2)^{\alpha-1}+\frac{1}{2\alpha}R^m_{jkl}\langle\psi^k,\nabla\phi^j\cdot\psi^l\rangle\\
&-\frac{1}{\alpha}F^m_\phi(\phi,\psi),
\end{split}
\end{equation}
\begin{equation}\label{elnonlinearalphalocal2}
\slashed\pt\psi^m=-\Gamma^m_{ij}\nabla\phi^i\cdot\psi^j+F^m_\psi(\phi,\psi).
\end{equation}

According to the proof in \cite{chen2005regularity}, $F_\psi$ will produce $\|\psi\|^3_{L^4}$. Therefore, the proof there cannot directly apply to our situation. To overcome it, we control the $L^\infty$-norm of $\psi$ first. Since $(\phi,\psi)\in\mathcal{F}^{\alpha,1/2}$, by the Sobolev embedding, we can prove that $\psi$ actually belongs to $W^{1,s}$ for any $2<s<2\alpha$.
\begin{lem}\label{psiC0gamma}
Let $(\phi,\psi)\in\mathcal{F}^{\alpha,1/2}$ be a weak solution of \eqref{elnonlinearalphalocal1} and \eqref{elnonlinearalphalocal2}. Suppose $F$ satisfies $(\rm F1)$ for some $p\leq2+2/\alpha$. Then $\psi\in C^{0,\gamma}(M,\Sigma M\otimes\phi^*TN)$ for any $0<\gamma<1-\frac{1}{\alpha}$.
\end{lem}
\begin{proof}
Since $M$ is closed, it is sufficient to prove the interior estimate. That is, we can apply the equation \eqref{elnonlinearalphalocal2} and the elliptic estimate for the first order equation (see \cite{chen2006dirac}) to the spinor multiplied by a cut-off function. For simplicity, we still use $\psi$ instead. The H\"older inequality implies
\begin{equation}
\|\nabla\phi\cdot\psi\|_{s_0}\leq\|\nabla\phi\|_{as_0}\|\psi\|_{\frac{as_0}{a-1}}
\end{equation}
for any $a,s_0>1$. We set
\begin{equation}
as_0=2\alpha \ \text{and} \  \frac{as_0}{a-1}=4.
\end{equation}
Then $4/3<s_0=\frac{4\alpha}{\alpha+2}\leq2$. By (F1), to make $F_\psi\in L^{s_0}$, we need $p\leq2+2/\alpha$. If $s_0=2$, that is, $\alpha=2$ and $p\leq3$, then the elliptic estimate for the first order equation (see \cite{chen2006dirac}) implies $\psi\in W^{1,2}\cap L^r$ for any $r>1$. If $s_0<2$, then, by Sobolev embedding and \eqref{elnonlinearalphalocal2}, we have $\psi\in L^{r_0}$, where $r_0=\frac{2s_0}{2-s_0}=\frac{4\alpha}{2-\alpha}>4$. Again, by setting
\begin{equation}
as_1=2\alpha \ \text{and} \  \frac{as_1}{a-1}=r_0,
\end{equation}
we have $s_1=\frac{2\alpha r_0}{2\alpha+r_0}$. To make $F_\psi\in L^{s_1}$, it is sufficient to set
$\frac{\alpha+2}{\alpha}s_1\leq r_0$, that is, $s_1\leq\frac{\alpha r_0}{\alpha+2}$. Since $r_0>4$, $\frac{2\alpha r_0}{2\alpha+r_0}<\frac{\alpha r_0}{\alpha+2}$. Now, we take $s_1=\frac{2\alpha r_0}{2\alpha+r_0}$. Then $s_1>s_0$ and
\begin{equation}
2-s_1=\frac{4\alpha+2r_0-2\alpha r_0}{2\alpha+r_0}=\frac{2(\alpha-1)}{2\alpha+r_0}\left(\frac{2\alpha}{\alpha-1}-r_0\right).
\end{equation} If $r_0=\frac{4\alpha}{2-\alpha}\geq\frac{2\alpha}{\alpha-1}$, that is, $\alpha\geq4/3$, we have $s_1\geq2$. Then $\psi\in W^{1,2}\cap L^r$ for any $r>1$. Otherwise, for $\alpha<4/3$, we get $s_1<2$. By Sobolev embedding, we have
\begin{equation}
r_0<r_1=\frac{2s_1}{2-s_1}=\frac{4\alpha r_0}{4\alpha+2r_0-2\alpha r_0}.
\end{equation}
Again, if $r_1\geq\frac{2\alpha}{\alpha-1}$, that is, $\alpha\geq6/5$, we have $s_2=\frac{2\alpha r_1}{2\alpha+r_1}\geq2$. Then $\psi\in W^{1,2}\cap L^r$ for any $r>1$. Otherwise, we get $r_2=\frac{2s_2}{2-s_2}$. Repeating this procedure, for each $\alpha$, after finitely many steps, we obtain $r_i\geq\frac{2\alpha}{\alpha-1}$. Then $\psi\in W^{1,2}\cap L^r$ for any $r>1$. By H\"older's inequality, $\psi\in W^{1,s}$ for any $2<s<2\alpha$. Then $\psi\in C^{0,\gamma}$ for any $0<\gamma<1-\frac{2}{s}<1-\frac{1}{\alpha}$. Thus, we complete the proof.

\end{proof}


Now, for completeness, we just follow the idea in \cite{chen2006dirac}. We need two lemmas. One is 
\begin{lem}\label{firstlemma}
Let $(\phi,\psi)$ be a weak solution of \eqref{elnonlinearalphalocal1} and \eqref{elnonlinearalphalocal2}. Suppose $F$ satisfies $(\rm F3)$ for some $q\geq0$. For any $\varepsilon>0$, there is a $\rho>0$ such that 
\begin{equation}
\begin{split}
\int_{B(x_1,\rho)}(1+|d\phi|^2)^{\alpha-1}|\nabla\phi|^2\eta^2
&\leq\varepsilon\int_{B(x_1,\rho)}(1+|d\phi|^2)^{\alpha-1}|\nabla\eta|^2+C\varepsilon\left(\int_{B(x_1,\rho)}|\psi|^4\eta^4\right)^{\frac12}\\
&\quad+C\varepsilon\int_{B(x_1,\rho)}(1+|\psi|^q)\eta^2,
\end{split}
\end{equation}
where $B(x_1,\rho)\subset\Omega$, $\eta\in W_0^{1,2\alpha}(B(x_1,\rho,\mathbb{R})$ and $C$ only depends on $N$ and the constant in $(\rm F3)$.
\end{lem}

\begin{proof}
Denote $G=(G^1,\cdots,G^n)$, where
\begin{equation}
G^m=\Gamma^m_{ij}\phi^i_\beta\phi^j_{\gamma}g^{\beta\gamma}(1+|d\phi|^2)^{\alpha-1}-\frac{1}{2\alpha}R^m_{jkl}\langle\psi^k,\nabla\phi^j\cdot\psi^l\rangle+\frac{1}{\alpha}F^m_\phi(\phi,\psi).
\end{equation}
Then the weak form of \eqref{elnonlinearalphalocal1} is
\begin{equation}\label{weakform1}
\int_\Omega(1+|d\phi|^2)^{\alpha-1}g^{\beta\gamma}\nabla_\beta\phi^i\nabla_\gamma\zeta^jg_{ij}=\int_\Omega G^i\zeta^jg_{ij}
\end{equation}
for any $\zeta\in W^{1,2\alpha}(\Omega,\mathbb{R}^L)$. Now, we choose $\zeta(x)=(\phi(x)-\phi(x_1))\eta^2$, then
\begin{equation}
\nabla_\beta\zeta^i=\nabla_\beta\phi^i\eta^2+2(\phi^i(x)-\phi^i(x_1))\eta\nabla_\beta\eta.
\end{equation}
Plugging this into \eqref{weakform1}, one gets
\begin{equation}\label{weakform2}
\begin{split}
&\int_{B(x_1,\rho)}G^i\zeta^jg_{ij}\\
&=\int_{B(x_1,\rho)}(1+|d\phi|^2)^{\alpha-1}\nabla_\beta\phi^i\nabla_\gamma\zeta^jg^{\beta\gamma}g_{ij}\\
&=\int_{B(x_1,\rho)}(1+|d\phi|^2)^{\alpha-1}\nabla_\beta\phi^ig_{ij}g^{\beta\gamma}(\nabla_\gamma\phi^j\eta^2+2(\phi^j(x)-\phi^j(x_1))\eta\nabla_\gamma\eta)\\
&=\int_{B(x_1,\rho)}(1+|d\phi|^2)^{\alpha-1}|\nabla\phi|^2\eta^2\\
&+2\int_{B(x_1,\rho)}(1+|d\phi|^2)^{\alpha-1}\eta(\phi^j(x)-\phi^j(x_1))g_{ij}\nabla_\beta\phi^i\nabla_\gamma\eta.
\end{split}
\end{equation}
Now, we estimate the left-hand side of \eqref{weakform2}.
\begin{equation}\label{left}
\begin{split}
\int_{B(x_1,\rho)}G^i\zeta^jg_{ij}
&=\int_{B(x_1,\rho)}\Gamma^i_{kl}\phi^k_\beta\phi^l_{\gamma}g^{\beta\gamma}g_{ij}(\phi^j(x)-\phi^j(x_1))(1+|d\phi|^2)^{\alpha-1}\eta^2\\
&\quad-\frac{1}{2\alpha}R^i_{klm}\langle\psi^l,\nabla\phi^k\cdot\psi^m\rangle(\phi^j(x)-\phi^j(x_1))g_{ij}\eta^2\\
&\quad+\frac{1}{\alpha}F^i_\phi(\phi,\psi)(\phi^j(x)-\phi^j(x_1))g_{ij}\eta^2\\
&\leq C_{N}\sup\limits_{B(x_1,\rho)}|\phi(x)-\phi(x_1)|\int_{B(x_1,\rho)}(1+|d\phi|^2)^{\alpha-1}|\nabla\phi|^2\eta^2\\
&\quad+C_{N}\sup\limits_{B(x_1,\rho)}|\phi(x)-\phi(x_1)|\left(\int_{B(x_1,\rho)}|\nabla\phi|^2\right)^{\frac12}\left(\int_{B(x_1,\rho)}|\psi|^4\eta^4\right)^{\frac12}\\
&\quad+C\sup\limits_{B(x_1,\rho)}|\phi(x)-\phi(x_1)|\int_{B(x_1,\rho)}(1+|\psi|^q)\eta^2.
\end{split}
\end{equation}
where we have used (F3). For the second term in the right-hand side of \eqref{weakform2}, we have
\begin{equation}\label{right2nd}
\begin{split}
&\left|2\int_{B(x_1,\rho)}(1+|d\phi|^2)^{\alpha-1}\eta(\phi^j(x)-\phi^j(x_1))g_{ij}g^{\beta\gamma}\nabla_\beta\phi^i\nabla_\gamma\eta\right|\\
&\leq2\sup\limits_{B(x_1,\rho)}|\phi(x)-\phi(x_1)|\int_{B(x_1,\rho)}(1+|d\phi|^2)^{\alpha-1}|\nabla\phi||\nabla\eta||\eta|\\
&\leq\frac12\int_{B(x_1,\rho)}(1+|d\phi|^2)^{\alpha-1}|\nabla\phi|^2\eta^2\\
&\quad+2\sup\limits_{B(x_1,\rho)}|\phi(x)-\phi(x_1)|^2\int_{B(x_1,\rho)}(1+|d\phi|^2)^{\alpha-1}|\nabla\eta|^2
\end{split}
\end{equation}
By estimates \eqref{left}, \eqref{right2nd} and \eqref{weakform2}, we have
\begin{equation}
\begin{split}
&\frac12\int_{B(x_1,\rho)}(1+|d\phi|^2)^{\alpha-1}|\nabla\phi|^2\eta^2
-2\sup\limits_{B(x_1,\rho)}|\phi(x)-\phi(x_1)|^2\int_{B(x_1,\rho)}(1+|d\phi|^2)^{\alpha-1}|\nabla\eta|^2\\
&\leq C_{N}\sup\limits_{B(x_1,\rho)}|\phi(x)-\phi(x_1)|\int_{B(x_1,\rho)}(1+|d\phi|^2)^{\alpha-1}|\nabla\phi|^2\eta^2\\
&\quad+C_{N}\sup\limits_{B(x_1,\rho)}|\phi(x)-\phi(x_1)|\left(\int_{B(x_1,\rho)}|\nabla\phi|^2\right)^{\frac12}\left(\int_{B(x_1,\rho)}|\psi|^4\eta^4\right)^{\frac12}\\
&\quad+C\sup\limits_{B(x_1,\rho)}|\phi(x)-\phi(x_1)|\int_{B(x_1,\rho)}(1+|\psi|^q)\eta^2.
\end{split}
\end{equation}
Since $\phi$ is continuous, we can choose $\rho$ so small  that $C_{N}\sup\limits_{B(x_1,\rho)}|\phi(x)-\phi(x_1)|<\frac14$. Then 
\begin{equation}
\begin{split}
\int_{B(x_1,\rho)}(1+|d\phi|^2)^{\alpha-1}|\nabla\phi|^2\eta^2
&\leq8\sup\limits_{B(x_1,\rho)}|\phi(x)-\phi(x_1)|^2\int_{B(x_1,\rho)}(1+|d\phi|^2)^{\alpha-1}|\nabla\eta|^2\\
&\quad+4C_{N}\sup\limits_{B(x_1,\rho)}|\phi(x)-\phi(x_1)|\|\phi\|_{1,2}\left(\int_{B(x_1,\rho)}|\psi|^4\eta^4\right)^{\frac12}\\
&\quad+4C\sup\limits_{B(x_1,\rho)}|\phi(x)-\phi(x_1)|\int_{B(x_1,\rho)}(1+|\psi|^q)\eta^2.
\end{split}
\end{equation}
Now, take $\rho$ still smaller so that $\sup\limits_{B(x_1,\rho)}|\phi(x)-\phi(x_1)|\max\{8,\|\phi\|_{1,2}\}\leq\varepsilon$. Then 
\begin{equation}
\begin{split}
\int_{B(x_1,\rho)}(1+|d\phi|^2)^{\alpha-1}|\nabla\phi|^2\eta^2
&\leq\varepsilon\int_{B(x_1,\rho)}(1+|d\phi|^2)^{\alpha-1}|\nabla\eta|^2\\
&\quad+C\varepsilon\left(\int_{B(x_1,\rho)}|\psi|^4\eta^4\right)^{\frac12}\\
&\quad+C\varepsilon\int_{B(x_1,\rho)}(1+|\psi|^q)\eta^2,
\end{split}
\end{equation}
where $C$ only depends on $N$ and the constant in (F3).

\end{proof}

The other lemma is 
\begin{lem}\label{secondlemma}
Let $\phi\in W^{1,2\alpha+2}\cap W^{3,2\alpha}(B(x_0,R),N)$ and $(\phi,\psi)$ be a weak solution of \eqref{elnonlinearalphalocal1} and \eqref{elnonlinearalphalocal2}. Suppose $F\in C^\infty$ satisfies $({\rm F}1)$ and $({\rm F}3)$ for some $p\leq2+2/\alpha$ and $q\geq0$. Then for $R$ small enough, we have
\begin{equation}\label{W22}
\|\nabla^2\phi\|_{L^2(B(x_0,\frac{R}{2}))}\leq C\|d\phi\|_{L^{2\alpha}(B(x_0,R))},
\end{equation}
where $C_1>0$ is a constant depending on $|\phi|_{C^0(M,N)}$, $R$, $\|\psi\|_{L^\infty(M,\Sigma M\otimes\phi^*TN)}$ and the constants in $({\rm F}1)$ and $({\rm F}3)$.
\end{lem}
\begin{proof}
Choose a local coordinate $\{y^i\}$ in a neighborhood $U$ of $\phi(x_0)$ such that $\Gamma^i_{jk}(\phi(x_0))=0$. Let $R$ be so small  that $B:=B(x_0,R)$ is mapped into $U$ by $\phi$. Since $\phi$ is continuous, we can choose $R$ small enough such that $|\phi(x)-\phi(x_0)|<\delta$ for a given $\delta$ to be determined later. Then \eqref{elnonlinearalphalocal2} implies
\begin{equation}
\begin{split}
|\slashed\pt\psi|&\leq|\Gamma^i_{jk}(\phi(x))-\Gamma^i_{jk}(\phi(x_0))||d\phi^j||\psi^k|+|F_\psi(\phi,\psi)|\\
&\leq C_N|\phi(x)-\phi(x_0)||d\phi||\psi|+|F_\psi(\phi,\psi)|\\
&\leq C_N\delta|d\phi||\psi|+|F_\psi(\phi,\psi)|
\end{split}
\end{equation}
Noting that $\slashed\pt(\psi\eta)=\eta\slashed\pt\psi+\nabla\eta\cdot\psi$, we get
\begin{equation}\label{slashed}
\begin{split}
&\|\slashed\pt(\psi\eta)\|_{L^{4/3}(B)}\\
&\leq\|\eta\slashed\pt\psi\|_{L^{4/3}(B)}+\|\nabla\eta\cdot\psi\|_{L^4/3(B)}\\
&\leq C_N\delta\||d\phi||\psi|\eta\|_{L^{4/3}(B)}+\|\eta|F_\psi(\phi,\psi)|\|_{L^{4/3}(B)}+\||\nabla\eta||\psi|\|_{L^{4/3}(B)}\\
&\leq C_N\delta\|d\phi\|_{L^2(B)}\||\psi|\eta\|_{L^4(B)}+\||\nabla\eta||\psi|\|_{L^{4/3}(B)}+C_1\|\eta(1+|\psi|^{p-1})\|_{L^{4/3}(B)}\\
&\leq C_N\delta\|d\phi\|_{L^2(B)}\||\psi|\eta\|_{L^4(B)}+\||\nabla\eta||\psi|\|_{L^{4/3}(B)}+C(C_1,\|\psi\|_{L^\infty})\|\eta\|_{L^{4/3}(B)}
\end{split}
\end{equation}
where we have replaced the constant $C$ in (F1) by $C_1$, and $C(C_1,\|\psi\|_{L^\infty})$ denotes a constant depending on $C_1$ and $\|\psi\|_{L^\infty(B)}$. 

  By the elliptic estimate for the first order equation (see \cite{chen2006dirac}), we have
\begin{equation}\label{1stelliptic}
\|\nabla(\psi\eta)\|_{L^{4/3}(B)}+\|\psi\eta\|_{L^{4}(B)}\leq C_R\|\slashed\pt(\psi\eta)\|_{L^{4/3}(B)}.
\end{equation}
Now, choose $R$ and $C_0$ so small  that $C_RC_N\delta\|d\phi\|_{L^2(B)}<\frac12$. We obtain from \eqref{slashed} and \eqref{1stelliptic} that 
\begin{equation}
\|\nabla(\psi\eta)\|_{L^{4/3}(B)}+\|\psi\eta\|_{L^{4}(B)}\leq C\||\nabla\eta||\psi|\|_{L^{4/3}(B)}+C(C_1,\|\psi\|_{L^\infty})\|\eta\|_{L^{4/3}(B)},
\end{equation}
which implies
\begin{equation}\label{key1}
\||\nabla\psi|\eta\|_{L^{4/3}(B)}+\|\psi\eta\|_{L^{4}(B)}\leq C\||\nabla\eta||\psi|\|_{L^{4/3}(B)}+C(C_1,\|\psi\|_{L^\infty})\|\eta\|_{L^{4/3}(B)}.
\end{equation}

For $\zeta\in W_0^{1,2\alpha}(B,\mathbb{R}^L)$, we have
\begin{equation}
\int_B(1+|d\phi|^2)^{\alpha-1}g^{\beta\gamma}\nabla_\beta\phi\nabla_\gamma\zeta=-\int_B{\rm div}((1+|d\phi|^2)^{\alpha-1}\nabla\phi)\zeta=\int_BG\zeta.
\end{equation}
Choosing $\zeta=g^{\sigma\tau}\nabla_\sigma(\xi^2\nabla_\tau\phi)$, where $\xi\in C_0^\infty(B,\mathbb{R})$ will  be determined later, we get
\begin{equation}\label{key2}
\begin{split}
&\int_Bg^{\beta\gamma}g^{\sigma\tau}\nabla_\sigma((1+|d\phi|^2)^{\alpha-1}\nabla_\beta\phi)\nabla_\gamma(\xi^2\nabla_\tau\phi)\\
&=-\int_Bg^{\beta\gamma}g^{\sigma\tau}(1+|d\phi|^2)^{\alpha-1}\nabla_\beta\phi\nabla_\sigma\nabla_\gamma(\xi^2\nabla_\tau\phi)\\
&=-\int_Bg^{\beta\gamma}g^{\sigma\tau}(1+|d\phi|^2)^{\alpha-1}\nabla_\beta\phi(\nabla_\gamma\nabla_\sigma(\xi^2\nabla_\tau\phi)-\xi^2(R^M)^\rho_{\sigma\gamma\tau}\nabla_\rho\phi)\\
&=-\int_BGg^{\sigma\tau}\nabla_\sigma(\xi^2\nabla_\tau\phi)-\int_Bg^{\beta\gamma}g^{\sigma\tau}(1+|d\phi|^2)^{\alpha-1}\xi^2\nabla_\beta\phi(R^M)^\rho_{\sigma\gamma\tau}\nabla_\rho\phi\\
&=\int_Bg^{\sigma\tau}\nabla_\sigma{G}\nabla_\tau\phi\xi^2-\int_Bg^{\beta\gamma}g^{\sigma\tau}(1+|d\phi|^2)^{\alpha-1}\nabla_\beta\phi(R^M)^\rho_{\sigma\gamma\tau}\nabla_\rho\phi\xi^2.
\end{split}
\end{equation}
Note that 
\begin{equation}\label{key2RHS1}
\begin{split}
&g^{\beta\gamma}g^{\sigma\tau}\nabla_\sigma((1+|d\phi|^2)^{\alpha-1}\nabla_\beta\phi)\nabla_\gamma(\xi^2\nabla_\tau\phi)\\
&=g^{\beta\gamma}g^{\sigma\tau}((1+|d\phi|^2)^{\alpha-1}\nabla_\sigma\nabla_\beta\phi+2(\alpha-1)(1+|d\phi|^2)^{\alpha-2}\langle\nabla_\sigma\nabla\phi,\nabla\phi\rangle\nabla_\beta\phi)\\
&\quad\cdot(\xi^2\nabla_\gamma\nabla_\tau\phi+2\xi\nabla_\gamma\xi\nabla_\tau\phi)\\
&=(1+|d\phi|^2)^{\alpha-1}|\nabla^2\phi|^2\xi^2+2\xi(1+|d\phi|^2)^{\alpha-1}g^{\beta\gamma}g^{\sigma\tau}\nabla_{\sigma\beta}^2\phi\nabla_\gamma\xi\nabla_\tau\phi\\
&\quad+2(\alpha-1)(1+|d\phi|^2)^{\alpha-2}g^{\beta\gamma}g^{\sigma\tau}\langle\nabla_\sigma\nabla\phi,\nabla\phi\rangle\nabla_\beta\phi\nabla_{\gamma\tau}^2\phi\xi^2\\
&\quad+4\xi(\alpha-1)(1+|d\phi|^2)^{\alpha-2}g^{\beta\gamma}g^{\sigma\tau}\langle\nabla_\sigma\nabla\phi,\nabla\phi\rangle\nabla_\beta\phi\nabla_\tau\phi\nabla_\gamma\xi.
\end{split}
\end{equation}
We control the last three terms as follows:
\begin{equation}
|2\xi(1+|d\phi|^2)^{\alpha-1}g^{\beta\gamma}g^{\sigma\tau}\nabla_{\sigma\beta}^2\phi\nabla_\gamma\xi\nabla_\tau\phi|\leq2(1+|d\phi|^2)^{\alpha-1}|d\phi||\xi\nabla\xi||\nabla^2\phi|,
\end{equation}
\begin{equation}
\begin{split}
&|2(\alpha-1)(1+|d\phi|^2)^{\alpha-2}g^{\beta\gamma}g^{\sigma\tau}\langle\nabla_\sigma\nabla\phi,\nabla\phi\rangle\nabla_\beta\phi\nabla_{\gamma\tau}^2\phi\xi^2|\\
&\leq2(\alpha-1)(1+|d\phi|^2)^{\alpha-2}|\nabla^2\phi|^2|\nabla\phi|^2\xi^2\\
&\leq2(\alpha-1)(1+|d\phi|^2)^{\alpha-1}|\nabla^2\phi|^2\xi^2,
\end{split}
\end{equation}
\begin{equation}
\begin{split}
&|4\xi(\alpha-1)(1+|d\phi|^2)^{\alpha-2}g^{\beta\gamma}g^{\sigma\tau}\langle\nabla_\sigma\nabla\phi,\nabla\phi\rangle\nabla_\beta\phi\nabla_\tau\phi\nabla_\gamma\xi|\\
&\leq4(\alpha-1)(1+|d\phi|^2)^{\alpha-2}|\nabla^2\phi||\nabla\phi|^3|\xi\nabla\xi|\\
&\leq4(\alpha-1)(1+|d\phi|^2)^{\alpha-1}|\nabla^2\phi||\nabla\phi||\xi\nabla\xi|.
\end{split}
\end{equation}
Plugging these estimates into \eqref{key2RHS1}, we have
\begin{equation}\label{key2RHS2}
\begin{split}
&g^{\beta\gamma}g^{\sigma\tau}\nabla_\sigma((1+|d\phi|^2)^{\alpha-1}\nabla_\beta\phi)\nabla_\gamma(\xi^2\nabla_\tau\phi)\\
&\geq(1-2(\alpha-1))(1+|d\phi|^2)^{\alpha-1}|\nabla^2\phi|^2\xi^2\\
&\quad-(2+4(\alpha-1))(1+|d\phi|^2)^{\alpha-1}|\nabla^2\phi||\nabla\phi||\xi\nabla\xi|\\
&\geq\frac12(1+|d\phi|^2)^{\alpha-1}|\nabla^2\phi|^2\xi^2-3(1+|d\phi|^2)^{\alpha-1}|\nabla^2\phi||\nabla\phi||\xi\nabla\xi|.
\end{split}
\end{equation}
For the integrands in the left-hand side of \eqref{key2}, we have the estimates
\begin{equation}\label{key2LHS2}
-\int_Bg^{\beta\gamma}g^{\sigma\tau}(1+|d\phi|^2)^{\alpha-1}\nabla_\beta\phi(R^M)^\rho_{\sigma\gamma\tau}\nabla_\rho\phi\xi^2\leq C_M(1+|d\phi|^2)^{\alpha-1}|\nabla\phi|^2\xi^2
\end{equation}
and 
\begin{equation}\label{key2LHS1}
\begin{split}
g^{\sigma\tau}\nabla_\sigma{G}\nabla_\tau\phi\xi^2&\leq|d\phi||\nabla G|\xi^2\\
&\leq C_N\xi^2|d\phi|\Bigg((1+|d\phi|^2)^{\alpha-1}|d\phi|^3+(1+|d\phi|^2)^{\alpha-1}|d\phi||\nabla^2\phi|\\
&\quad\quad+(\alpha-1)(1+|d\phi|^2)^{\alpha-2}|d\phi|^3|\nabla^2\phi|+|d\phi|^2|\psi|^2+|d\phi||\nabla\psi||\psi|\\
&\quad\quad+|\nabla^2\phi||\psi|^2\Bigg)+C_3(1+|\psi|^q)|d\phi|^2\xi^2+C_3(1+|\psi|^{q-1})|d\phi||\nabla\psi|\xi^2,
\end{split}
\end{equation}
where we have replaced the constant $C$ in (F3) by $C_3$ and take $C_N>1$.

By \eqref{key2RHS2}, \eqref{key2LHS2} and \eqref{key2LHS1}, \eqref{key2} becomes
\begin{equation}\label{key2'}
\begin{split}
&\int_B(1+|d\phi|^2)^{\alpha-1}|\nabla^2\phi|^2\xi^2\\
&\leq6\int_B(1+|d\phi|^2)^{\alpha-1}|\nabla^2\phi||\nabla\phi||\xi\nabla\xi|+C_M\int_B(1+|d\phi|^2)^{\alpha-1}|\nabla\phi|^2\xi^2\\
&\quad+C_N\int_B(1+|d\phi|^2)^{\alpha-1}|d\phi|^4\xi^2+C_N\int_B(1+|d\phi|^2)^{\alpha-1}|d\phi|^2|\nabla^2\phi|\xi^2\\
&\quad+C_N\int_B|d\phi|^3|\psi|^2\xi^2+C_N\int_B|d\phi|^2|\nabla\psi||\psi|\xi^2+C_N\int_B|\nabla^2\phi||d\phi||\psi|^2\xi^2\\
&\quad+C(C_3,\|\psi\|_{L^\infty})\int_B|d\phi|^2\xi^2+C(C_3,\|\psi\|_{L^\infty})\int_B|d\phi||\nabla\psi|\xi^2\\
&=:I+II+III+IV+V+VI+VII+VIII+IX
\end{split}
\end{equation}
For small $\varepsilon_1>0$, we have
\begin{equation}\label{I}
I\leq6\varepsilon_1\int_B(1+|d\phi|^2)^{\alpha-1}|\nabla^2\phi|^2\xi^2+\frac{6}{\varepsilon_1}\int_B(1+|d\phi|^2)^{\alpha-1}|d\phi|^2|\nabla\xi|^2
\end{equation}
and 
\begin{equation}\label{IV}
IV\leq C_N\varepsilon_1\int_B(1+|d\phi|^2)^{\alpha-1}|\nabla^2\phi|^2\xi^2+\frac{C_N}{\varepsilon_1}\int_B(1+|d\phi|^2)^{\alpha-1}|d\phi|^4\xi^2.
\end{equation}
Choosing $\eta=|d\phi|\xi$ in \eqref{key1}, we obtain
\begin{equation}\label{key1'}
\begin{split}
\||\nabla\psi||d\phi|\xi\|_{L^{4/3}(B)}+\|\psi|d\phi|\xi\|_{L^{4}(B)}
&\leq C\||\nabla(|d\phi|\xi)||\psi|\|_{L^{4/3}(B)}\\&\quad+C(C_1,\|\psi\|_{L^\infty})\||d\phi|\xi\|_{L^{4/3}(B)}.
\end{split}
\end{equation}
Now, we control the remaining terms in the right-hand side of \eqref{key2'} as follows:
\begin{equation}
\begin{split}
V&\leq C_N\varepsilon_1\int_B|d\phi|^2|\psi|^4\xi^2+\frac{C_N}{\varepsilon_1}\int_B|d\phi|^4\xi^2\\
&\leq C_N\varepsilon_1\left(\int_B|\psi|^4\right)^{1/2}\left(\int_B|\psi|^4|d\phi|^4\xi^4\right)^{1/2}+\frac{C_N}{\varepsilon_1}\int_B|d\phi|^4\xi^2\\
&\leq 2C_NC^2\varepsilon_1\left(\int_B|\psi|^4\right)^{1/2}\||\nabla(|d\phi|\xi)||\psi|\|^2_{L^{4/3}(B)}\\
&\quad+2C_NC^2(C_1,\|\psi\|_{L^\infty})\varepsilon_1\left(\int_B|\psi|^4\right)^{1/2}\||d\phi|\xi\|^2_{L^{4/3}(B)}+\frac{C_N}{\varepsilon_1}\int_B|d\phi|^4\xi^2,
\end{split}
\end{equation}
where we denote $(C(C_1,\|\psi\|_{L^\infty}))^2$ by $C^2(C_1,\|\psi\|_{L^\infty})$. By 
\begin{equation}\label{key3}
\begin{split}
&\||\nabla(|d\phi|\xi)||\psi|\|^2_{L^{4/3}(B)}\\
&\leq2\||\nabla^2\phi|\psi\xi\|^2_{L^{4/3}(B)}+2\||d\phi||\nabla\xi||\psi|\|^2_{L^{4/3}(B)}\\
&\leq2\Bigg(\int_B|\nabla^2\phi|^2\xi^2\Bigg)\Bigg(\int_B|\psi|^4\Bigg)^{1/2}+2\Bigg(\int_B|d\phi|^2|\nabla\xi|^2\Bigg)\Bigg(\int_B|\psi|^4\Bigg)^{1/2}\\
&=2\Bigg(\int_B|\psi|^4\Bigg)^{1/2}\Bigg(\int_B|\nabla^2\phi|^2\xi^2+\int_B|d\phi|^2|\nabla\xi|^2\Bigg),
\end{split}
\end{equation}
 we have
 \begin{equation}\label{V}
 \begin{split}
V&\leq4C_NC^2\varepsilon_1\Bigg(\int_B|\psi|^4\Bigg)\Bigg(\int_B|\nabla^2\phi|^2\xi^2+\int_B|d\phi|^2|\nabla\xi|^2\Bigg)\\
&\quad+2C_NC^2(C_1,\|\psi\|_{L^\infty})\varepsilon_1\left(\int_B|\psi|^4\right)^{1/2}\||d\phi|\xi\|^2_{L^{4/3}(B)}+\frac{C_N}{\varepsilon_1}\int_B|d\phi|^4\xi^2,
\end{split}
\end{equation}
 \begin{equation}\label{VI}
 \begin{split}
VI&\leq C_N\||\nabla\psi||d\phi|\xi\|_{L^{4/3}(B)}\|\psi|d\phi|\xi\|_{L^{4}(B)}\\
&\leq 2C_NC^2\||\nabla(|d\phi|\xi)||\psi|\|^2_{L^{4/3}(B)}+2C_NC^2(C_1,\|\psi\|_{L^\infty})\||d\phi|\xi\|^2_{L^{4/3}(B)}\\
&\leq 4C_NC^2\Bigg(\int_B|\psi|^4\Bigg)^{1/2}\Bigg(\int_B|\nabla^2\phi|^2\xi^2+\int_B|d\phi|^2|\nabla\xi|^2\Bigg)\\
&\quad+2C_NC^2(C_1,\|\psi\|_{L^\infty})\||d\phi|\xi\|^2_{L^{4/3}(B)}
\end{split}
\end{equation} 
\begin{equation}\label{VII}
\begin{split}
VII&\leq\frac12\int_B|\nabla^2\phi|^2\xi^2+\frac{C_N^2}{2}\int_B|d\phi|^2|\psi|^4\xi^2\\
&\leq\frac12\int_B|\nabla^2\phi|^2\xi^2+{C_N^2}C^2\left(\int_B|\psi|^4\right)^{1/2}\||\nabla(|d\phi|\xi)||\psi|\|^2_{L^{4/3}(B)}\\
&\quad+C_N^2C^2(C_1,\|\psi\|_{L^\infty})\left(\int_B|\psi|^4\right)^{1/2}\||d\phi|\xi\|^2_{L^{4/3}(B)}\\
&\leq\frac12\int_B|\nabla^2\phi|^2\xi^2+2{C_N^2}C^2\left(\int_B|\psi|^4\right)\Bigg(\int_B|\nabla^2\phi|^2\xi^2+\int_B|d\phi|^2|\nabla\xi|^2\Bigg)\\
&\quad+C_N^2C^2(C_1,\|\psi\|_{L^\infty})\left(\int_B|\psi|^4\right)^{1/2}\||d\phi|\xi\|^2_{L^{4/3}(B)}
\end{split}
\end{equation}
and
\begin{equation}\label{IX}
\begin{split}
IX&\leq C(C_3,\|\psi\|_{L^\infty})R^{1/2}\||d\phi||\nabla\psi|\xi\|_{L^{4/3}(B)}\\
&\leq C(C_3,\|\psi\|_{L^\infty})R^{1/2}+C(C_3,\|\psi\|_{L^\infty})R^{1/2}\||d\phi||\nabla\psi|\xi\|_{L^{4/3}(B)}^2\\
&\leq2C(C_3,\|\psi\|_{L^\infty})R^{1/2}\Bigg(\int_B|\psi|^4\Bigg)^{1/2}\Bigg(\int_B|\nabla^2\phi|^2\xi^2+\int_B|d\phi|^2|\nabla\xi|^2\Bigg)\\
&\quad+C(C_3,\|\psi\|_{L^\infty})R^{1/2},
\end{split}
\end{equation}
where we have used $0\leq\xi\leq1$.
Plugging these estimates into \eqref{key2'}
and choosing $\varepsilon_1$ and $R$ so small that 
\begin{equation}
\max\{4C_NC^2,2C(C_3,\|\psi\|_{L^\infty})\}(1+\|\psi\|_{L^4(B)}^4)<\frac{1}{12} \ \text{and} \ (6+C_N)\varepsilon_1<\frac{1}{12},
\end{equation}
 we get
 \begin{equation}
\begin{split}
&\frac{1}{12}\int_B(1+|d\phi|^2)^{\alpha-1}|\nabla^2\phi|^2\xi^2\\
&\leq\frac{6}{\varepsilon_1}\int_B(1+|d\phi|^2)^{\alpha-1}|d\phi|^2|\nabla\xi|^2+\frac{C_N}{\varepsilon_1}\int_B(1+|d\phi|^2)^{\alpha-1}|d\phi|^4\xi^2\\
&\quad+C_M\int_B(1+|d\phi|^2)^{\alpha-1}|\nabla\phi|^2\xi^2+C_N\int_B(1+|d\phi|^2)^{\alpha-1}|d\phi|^4\xi^2\\
&\quad+\frac{1}{3}\int_B|d\phi|^2|\nabla\xi|^2+\frac{C_N}{\varepsilon_1}\int_B|d\phi|^4\xi^2+5C_N^2C^2(C_1,\|\psi\|_{L^\infty})\|d\phi\|^2_{L^{4/3}(B)}\\
&\quad+C(C_3,\|\psi\|_{L^\infty})(1+\|d\phi\|_{L^2}^2).
\end{split}
\end{equation}
Therefore, we obtain
 \begin{equation}\label{keyfinal}
\begin{split}
&\int_B(1+|d\phi|^2)^{\alpha-1}|\nabla^2\phi|^2\xi^2\\
&\leq C\int_B(1+|d\phi|^2)^{\alpha-1}|d\phi|^2|\nabla\xi|^2+C\int_B(1+|d\phi|^2)^{\alpha-1}|d\phi|^4\xi^2+C\int_B|d\phi|^{2\alpha}+C.
\end{split}
\end{equation}

Now, for $\varepsilon>0$, let $\rho>0$ be as in Lemma \ref{firstlemma}, and choose a cut-off function $\xi\in C^\infty_0(B(x_1,\rho))\subset B(x_0,R)$ such that
$$0\leq\xi\leq1, \ \xi=1\ \text{in} \ B(x_1,\rho/2), \ |\nabla\xi|\leq\frac{4}{\rho} \ \text{in} \ B(x_1,\rho).$$
Denoting $B_\rho:=B(x_1,\rho)$ for simplicity, we get
\begin{equation}
\begin{split}
&\int_{B_\rho}(1+|d\phi|^2)^{\alpha-1}|d\phi|^4\xi^2\\
&=\int_{B_\rho}(1+|d\phi|^2)^{\alpha-1}|d\phi|^2(|d\phi|\xi)^2\\
&\leq\varepsilon\int_{B_\rho}(1+|d\phi|^2)^{\alpha-1}|\nabla(|d\phi|\xi)|^2+C\varepsilon\left(\int_{B_\rho}|\psi|^4(|d\phi|\xi)^4\right)^{\frac12}\\
&\quad+C\varepsilon\int_{B_\rho}(1+|\psi|^q)(|d\phi|\xi)^2\\
&\leq2\varepsilon\int_{B_\rho}(1+|d\phi|^2)^{\alpha-1}|\nabla^2\phi|^2\xi^2+2\varepsilon\int_{B_\rho}(1+|d\phi|^2)^{\alpha-1}|d\phi|^2|\nabla\xi|^2\\
&\quad+C\varepsilon\left(\int_{B_\rho}|\psi|^4|d\phi|^4\xi^4\right)^{\frac12}+C\varepsilon\int_{B_\rho}|d\phi|^2.
\end{split}
\end{equation}
It follows from \eqref{key1'} and \eqref{key3} that 
\begin{equation}
\begin{split}
\||\psi||d\phi|\xi\|_{L^4(B_\rho)}^2&\leq2C^2\||\nabla(|d\phi|\xi)||\psi|\|^2_{L^{4/3}(B_\rho)}+2C^2(C_1,\|\psi\|_{L^\infty})\||d\phi|\xi\|^2_{L^{4/3}(B_\rho)}\\
&\leq4C^2\Bigg(\int_{B_\rho}|\psi|^4\Bigg)^{1/2}\Bigg(\int_{B_\rho}|\nabla^2\phi|^2\xi^2+\int_{B_\rho}|d\phi|^2|\nabla\xi|^2\Bigg)\\
&\quad+2C^2(C_1,\|\psi\|_{L^\infty})\||d\phi|\xi\|^2_{L^{4/3}(B_\rho)}.
\end{split}
\end{equation}
Therefore,
\begin{equation}\label{2ndkeyfinal}
\begin{split}
&\int_{B_\rho}(1+|d\phi|^2)^{\alpha-1}|d\phi|^4\xi^2\\
&\leq2\varepsilon\int_{B_\rho}(1+|d\phi|^2)^{\alpha-1}|\nabla^2\phi|^2\xi^2+2\varepsilon\int_{B_\rho}(1+|d\phi|^2)^{\alpha-1}|d\phi|^2|\nabla\xi|^2\\
&\quad+C\varepsilon\Bigg(\int_{B_\rho}|\nabla^2\phi|^2\xi^2+\int_{B_\rho}|d\phi|^2|\nabla\xi|^2\Bigg)\\
&\quad+C\varepsilon\||d\phi|\xi\|^2_{L^{4/3}(B_\rho)}+C\varepsilon\int_{B_\rho}|d\phi|^2\\
&\leq C\varepsilon\Bigg(\int_{B_\rho}(1+|d\phi|^2)^{\alpha-1}|\nabla^2\phi|^2\xi^2+\int_{B_\rho}(1+|d\phi|^2)^{\alpha-1}|d\phi|^2|\nabla\xi|^2\Bigg)\\
&\quad+C\int_{B_\rho}|d\phi|^{2\alpha}+C,
\end{split}
\end{equation}
where the constants $C$ are different between lines as before. Substituting \eqref{2ndkeyfinal} into \eqref{keyfinal}, we have
\begin{equation}
\int_{B_\rho}(1+|d\phi|^2)^{\alpha-1}|\nabla^2\phi|^2\xi^2\leq C\int_{B_\rho}(1+|d\phi|^2)^{\alpha-1}|d\phi|^2|\nabla\xi|^2+C\int_{B_\rho}|d\phi|^{2\alpha}+C,
\end{equation}
hence, 
\begin{equation}\label{w22}
\begin{split}
\int_{B_{\rho/2}}(1+|d\phi|^2)^{\alpha-1}|\nabla^2\phi|^2&\leq\frac{C}{\rho^2}\int_{B_\rho}(1+|d\phi|^2)^{\alpha-1}|d\phi|^2+C\int_{B_\rho}|d\phi|^{2\alpha}+C\\
&\leq\frac{C}{\rho^2}\int_{B_\rho}(1+|d\phi|^2)^{\alpha}+C.
\end{split}
\end{equation}
Covering $B(x_0,\frac{R}{2})$ with $B(x_1,\frac{\rho}{2})$ and using \eqref{w22} we obtain \eqref{W22}. 

\end{proof}

Now we can prove the main theorem in this section.
\begin{proof}[Proof of Theorem \ref{regularity}]
First we show that $\phi\in W^{2,2}\cap W^{1,4}(B(x_0,\frac{R}{2}),N)$. This can be done just by replacing weak derivatives by difference quotients in the proof of Lemma \ref{secondlemma}. Denote 
\begin{equation}
\Delta^h_i\phi(x):=\frac{\phi(x+hE_i)-\phi(x)}{h},
\end{equation}
where $(E_1,\cdots,E_L)$ is an orthonormal basis of $\mathbb{R}^L$ and $h\in\mathbb{R}$. $\Delta^h=(\delta^h_1,\cdots,\delta^h_L)$. 
Similar to \eqref{keyfinal}, we have
 \begin{equation}\label{quotient}
\begin{split}
&\int_{B(x_0,R)}|\nabla(\Delta^h\phi)|^2\xi^2\\
&\leq C\int_{B(x_0,R)}(1+|\Delta^h\phi|^2)^{\alpha}|\nabla\xi|^2+C\int_{B(x_0,R)}(1+|d\phi|^2)^{\alpha-1}|\nabla\phi|^2|\Delta^h\phi|^2\xi^2\\
&\quad+C\int_{B(x_0,R)}|d\phi|^{2\alpha}+C.
\end{split}
\end{equation}
Using 
\begin{equation}
C\int_{B(x_0,R)}(1+|\Delta^h\phi|^2)^{\alpha}|\nabla\xi|^2\leq C\int_{B(x_0,R)}(1+|d\phi|^2)^{\alpha}|\nabla\xi|^2
\end{equation}
and applying the Lemma \ref{firstlemma} to the second term in the right-hand side of \eqref{quotient}, we obtain 
\begin{equation}
\int_{B(x_1,\frac{\rho}{2})}|\nabla(\Delta^h\phi)|^2\xi^2\leq\frac{C}{\rho^2}\int_{B(x_1,\rho)}(1+|d\phi|^2)^\alpha+C.
\end{equation}
Then by Lemma A.2.2 in \cite{jost2017riemannian} implies the weak derivative $\nabla^2\phi$ exists and \eqref{W22} still holds.

Since $\phi\in W^{2,2}$, we have $\phi\in W^{1,p}$ for any $p>0$. Together with the Lemma \ref{psiC0gamma} and the equation \eqref{elnonlinearalphalocal1}, we know $\phi\in W^{2,p}$ for any $p>0$. Thus, $\phi\in C^{1,\gamma}$. By the elliptic estimates for the equation \eqref{elnonlinearalphalocal2}, we have $\psi\in C^{1,\gamma}$. Then the standard arguments yield that both $\phi$ and $\psi$ are smooth. This completes the proof.

\end{proof}

\section{$\alpha$-Dirac-harmonic maps}
In this section, we want to approximate the $\alpha$-Dirac-harmonic maps by the perturbed $\alpha$-Dirac-harmonic maps. More precisely, by Theorem \ref{nontrivial} and Theorem \ref{regularity}, we have a sequence of perturbed $\alpha$-Dirac-harmonic maps $(\phi_k,\psi_k)$, which are the critical points of the  functionals
\begin{equation}\label{Lk}
\mathcal{L}_k^{\alpha}(\phi,\psi)=\frac12\int_{M}(1+|d\phi|^2)^{\alpha}+\frac12\int_{M}\langle\psi,\slashed{D}\psi\rangle_{\Sigma M\otimes\phi^*TN}-\frac1k\int_{M}F(\phi,\psi),
\end{equation}
where $F$ satisfies the assumptions in Theorem \ref{nontrivial} and Theorem \ref{regularity}. If $(\phi_n,\psi_n)$ converges to $(\phi,\psi)$ smoothly, then we get the existence of nontrivial $\alpha$-Dirac-harmonic maps. We now come to the  $\varepsilon$-regularity which we shall need  to investigate the convergence.
\begin{thm}\label{epsilonregularity}
Suppose $F$ satisfies\\
$(\rm F6)$ \ $|F_\psi(\phi,\psi)|\leq C|\psi|^{r-1} \ \text{for} \ 3<r\leq2+2/\alpha,$\\
$(\rm F7)$ \ $|F_\phi(\phi,\psi)|\leq C|\psi|^{q} \ \text{for} \ q\geq0.$

There is $\varepsilon_0>0$ and $\alpha_0>0$ such that if $(\phi,\psi):(D,g_{\beta\gamma})\to(N,g_{ij})$ is a smooth perturbed $\alpha$-Dirac-harmonic map satisfying
\begin{equation}\label{epsilon}
\int_M(|d\phi|^{2\alpha}+|\psi|^4)\leq\Lambda<+\infty\ \text{and} \ \int_D|d\phi|^2\leq\varepsilon_0 
\end{equation}
for $1\leq\alpha<\alpha_0$, then we have
\begin{equation}
\|d\phi\|_{\tilde{D},1,p}+\|\psi\|_{\tilde{D},1,p}\leq C(D,N,\Lambda,p)(\|d\phi\|_{D,0,2}+\|\psi\|_{D,0,4}),
\end{equation}
and
\begin{equation}\label{boundpsi}
\|\psi\|_{L^\infty(\tilde{D})}\leq C(D,N,\Lambda)\|\psi\|_{D,0,4}
\end{equation}
for any $\tilde D\subset D$ and $p>1$, where $C(D,N,\Lambda,p)$ denotes a constant depending on $D,N,\Lambda$ and $p$.

\end{thm}

Note that \eqref{boundpsi} is a consequence of the proof of Lemma \ref{psiC0gamma} and (F6). This is different from the proof in \cite{chen2005regularity}. Our proof of the theorem above is based on the ones in \cite{chen2005regularity} and \cite{lin2008analysis}. For the proof of the theorem, we need alemma.
\begin{lem}\label{lemphi1,4}
There are $\varepsilon_0>0$ and $\alpha_0>0$ such that if $(\phi,\psi)$ and $F$ satisfy \eqref{epsilon} and $(\rm F6)$, $(\rm F7)$, respectively, then 
\begin{equation}\label{phi1,4}
\|\phi\|_{D_1,1,4}\leq C(N,D_1)(\|d\phi\|_{D,0,2}+\|\psi\|_{D,0,4}^2+\|\psi\|_{D,0,4}^q).
\end{equation}
\end{lem} 
\begin{proof}
Choose a cut-off function $\eta\in[0,1]$ with $\eta|_{D_1}\equiv1$ and $\rm{Supp}(\eta)\subset D$. As before, $(\phi,\psi)$ locally satisfies the system
\begin{equation}\label{localphi}
\begin{split}
{\rm div}((1+|d\phi|^2)^{\alpha-1}\nabla\phi^m)=&-\Gamma^m_{ij}\phi^i_\beta\phi^j_{\gamma}g^{\beta\gamma}(1+|d\phi|^2)^{\alpha-1}+\frac{1}{2\alpha}R^m_{jkl}\langle\psi^k,\nabla\phi^j\cdot\psi^l\rangle\\
&-\frac{1}{\alpha}F^m_\phi(\phi,\psi),
\end{split}
\end{equation}
\begin{equation}\label{localpsi}
\slashed\pt\psi^m=-\Gamma^m_{ij}\nabla\phi^i\cdot\psi^j+F^m_\psi(\phi,\psi).
\end{equation}
 \eqref{localphi} implies
\begin{equation}
\begin{split}
\Bigg|\Delta\phi^m+(\alpha-1)\frac{\langle\nabla^2\phi,\nabla\phi\rangle\nabla\phi^m}{1+|d\phi|^2}\Bigg|&\leq|\Gamma^m_{ij}\phi^i_\beta\phi^j_{\gamma}g^{\beta\gamma}|+\frac{1}{2\alpha}|R^m_{jkl}\langle\psi^k,\nabla\phi^j\cdot\psi^l\rangle|\\
&\quad+\frac{1}{\alpha}|F^m_\phi(\phi,\psi)|.
\end{split}
\end{equation}
Then 
\begin{equation}
|\eta\Delta\phi|\leq(\alpha-1)|\eta\nabla^2\phi|+C_N|\nabla\phi\nabla\phi\eta|+\frac{1}{2\alpha}C_N|\psi|^2|\nabla\phi\eta|+C|\psi|^q\eta.
\end{equation}
Since
\begin{equation}
|\eta\nabla^2\phi|=|\nabla^2(\eta\phi)-\nabla\eta\nabla\phi-\nabla^2\eta\phi|\leq|\nabla^2(\eta\phi)|+C(|\phi|+|d\phi|)
\end{equation}
and 
\begin{equation}
|\nabla\phi\eta|=|\nabla(\phi\eta)-\phi\nabla\eta|\leq|\nabla(\phi\eta)|+C|\phi|,
\end{equation}
we have
\begin{equation}
\begin{split}
|\eta\Delta\phi|&\leq(\alpha-1)|\nabla^2(\eta\phi)|+C_N|\nabla\phi\nabla(\phi\eta)|+\frac{C_N}{2}|\psi|^2|\nabla\phi\eta|\\
&\quad+C|\psi|^q\eta+C(|\phi|+|d\phi|).
\end{split}
\end{equation}
Therefore, 
\begin{equation}
\begin{split}
|\Delta(\eta\phi)|&\leq(\alpha-1)|\nabla^2(\eta\phi)|+C_N|\nabla\phi\nabla(\phi\eta)|+C_N|\psi|^2|\nabla\phi\eta|\\
&\quad+C|\psi|^q\eta+C(|\phi|+|d\phi|).
\end{split}
\end{equation}
Thus, for any $p>1$, we get
\begin{equation}\label{2,p}
\begin{split}
\|\Delta(\eta\phi)\|_{D,0,p}&\leq(\alpha-1)\|\eta\phi\|_{D,2,p}+C_N\|\nabla\phi\nabla(\phi\eta)\|_{D,0,p}+C_N\||\psi|^2\nabla\phi\eta\|_{D,0,p}\\
&\quad+C\||\psi|^q\eta\|_{D,0,p}+C\|\phi\|_{D,1,p}.
\end{split}
\end{equation}
Without loss of generality, we assume $\int_D\phi=0$ so that $\|\phi\|_{D,1,p}\leq C\|d\phi\|_{D,0,p}$. By taking $p=4/3$ in the inequality above and H\"older's inequality, we obtain
\begin{equation}
\|\nabla\phi\nabla(\phi\eta)\|_{D,0,4/3}\leq\|\nabla\phi\|_{D,0,2}\|\nabla(\phi\eta)\|_{D,0,4},
\end{equation}
\begin{equation}
\||\psi|^2\nabla\phi\eta\|_{D,0,4/3}\leq\|\psi\|_{D,0,4}^2\|\nabla\phi\|_{D,0,2}
\end{equation}
and
\begin{equation}
\||\psi|^q\eta\|_{D,0,4/3}\leq\|\psi\|_{D,0,4}^q.
\end{equation}

On the other hand, let $c(p)$ be the operator norm of $\Delta^{-1}:L^p(D)\to W^{2,p}(D)\cap W^{1,2}_0(D)$. Then we get
\begin{equation}
\|\Delta(\eta\phi)\|_{D,0,p}\geq(c(p))^{-1}\|\eta\phi\|_{D,2,p}.
\end{equation}
Plugging the above estimates into \eqref{2,p}, we have
\begin{equation}
\begin{split}
&((c(4/3))^{-1}-(\alpha-1))\|\eta\phi\|_{D,2,4/3}\\
&\leq C(N,D_1)(\|\nabla\phi\|_{D,0,2}\|\phi\eta\|_{D,1,4}+\|\psi\|_{D,0,4}^2\|d\phi\|_{D,0,2}+\|\psi\|_{D,0,4}^q+\|d\phi\|_{D,0,4/3}).
\end{split}
\end{equation}
Together with the Sobolev inequality $\|\phi\eta\|_{D,1,4}\leq C_1\|\phi\eta\|_{D,2,4/3}$, this implies
\begin{equation}
\begin{split}
&(((c(4/3))^{-1}-(\alpha-1))C_1^{-1}-\sqrt\varepsilon_0C(N,D_1))\|\eta\phi\|_{D,1,4}\\
&\leq C(N,D_1)(\|\psi\|_{D,0,4}^2+\|d\phi\|_{D,0,2}+\|\psi\|_{D,0,4}^q).
\end{split}
\end{equation}
Now we choose $\alpha_0,\varepsilon_0>0$ such that 
\begin{equation}
(c(4/3))^{-1}-(\alpha_0-1)>\frac12(c(4/3))^{-1}
\end{equation}
and
\begin{equation}
\varepsilon_0\leq\frac12\frac{((c(4/3))^{-1}-(\alpha_0-1))C_1^{-1}}{C(N,D_1)},
\end{equation}
then we get \eqref{phi1,4}. Thus we complete the proof.

\end{proof}

\begin{proof}[Proof of Theorem \ref{epsilonregularity}]
Choose $\tilde{D}\subset D_2\subset D_1\subset D$ and a cut-off function $\eta\in[0,1]$ with $\eta|_{D_2}\equiv1$ and ${\rm Supp}(\eta)\subset D_1$. Taking $p=2$ in \eqref{2,p} on $D_1$(we temporarily assume $\int_{D_1}\phi=0$), we have
\begin{equation}
\begin{split}
\|\Delta(\eta\phi)\|_{D_1,0,p}&\leq(\alpha-1)\|\eta\phi\|_{D_1,2,2}+C_N\|\nabla\phi\nabla(\phi\eta)\|_{D_1,0,2}+C_N\||\psi|^2\nabla\phi\eta\|_{D_1,0,2}\\
&\quad+C\||\psi|^q\eta\|_{D_1,0,2}+C\|\phi\|_{D_1,1,2}.
\end{split}
\end{equation}
As in the proof of Lemma \ref{lemphi1,4}, we get
\begin{equation}\label{phi2,2}
\begin{split}
\|\eta\phi\|_{D_1,2,2}&\leq C(N,D_1)(\|\nabla\phi\|_{D_1,0,4}^2+\|\psi\|_{D_1,0,4}^2\|\nabla\phi\|_{D_1,0,4}\\
&\quad+\||\psi|^q\eta\|_{D_1,0,2}+\|\phi\|_{D_1,1,2})\\
&\leq C(N,D_1)(\|\nabla\phi\|_{D_1,0,4}^2+\|\psi\|_{D_1,0,4}^4+\|\psi\|_{D_1,0,4}^q+\|\phi\|_{D_1,1,2})\\
&\leq C(N,D_1)(\|\phi\|_{D_1,1,4}^2+\|\psi\|_{D_1,0,4}^4+\|\psi\|_{D_1,0,4}^q+\|\phi\|_{D_1,1,4}).
\end{split}
\end{equation}
By Lemma \ref{lemphi1,4}, we obtain
\begin{equation}
\|\phi\|_{D_2,2,2}\leq C(N,D_1)(\|\psi\|_{D_1,0,4}^2+\|\psi\|_{D_1,0,4}^4+\|\psi\|_{D_1,0,4}^q+\|\psi\|_{D_1,0,4}^{2q}+\|d\phi\|_{D_1,0,2}).
\end{equation}
The Sobolev inequality gives us 
\begin{equation}\label{dphi0p}
\begin{split}
\|d\phi\|_{D_2,0,p}&\leq C(N,D_1)(\|\psi\|_{D_1,0,4}^2+\|\psi\|_{D_1,0,4}^4+\|\psi\|_{D_1,0,4}^q+\|\psi\|_{D_1,0,4}^{2q}+\|d\phi\|_{D_1,0,2})\\
&\leq C(D,N,\Lambda,p)(\|d\phi\|_{D,0,2}+\|\psi\|_{D,0,4})
\end{split}
\end{equation}
for any $p>1$. This also holds for $\phi$ without $\int_{D_1}\phi=0$.

Now, since $\phi\in W^{1,p}$ and $\psi\in L^p$ for any $p>1$, the estimates in the theorem follow from the standard $L^p$-estimate for the Dirac operator and the $W^{2,p}$-estimate for the Laplace operator immediately.

\end{proof}

With the $\varepsilon$-regularity in hand, one can easily  prove the  following theorem.
\begin{thm}
Let $(\phi_k,\psi_k)$ be  smooth critical points of the  functional $L^\alpha_k$ in \eqref{Lk} with uniformly bounded energy:
\begin{equation}
E_\alpha(\phi_k,\psi_k;M):=\int_M(|d\phi|^{2\alpha}+|\psi|^4)\leq\Lambda<+\infty.
\end{equation} 
Suppose F satisfies $(\rm F6)$ and $(\rm F7)$. Then there are a subsequence, still denoted by $\{(\phi_k,\psi_k)\}$, and a smooth $\alpha$-Dirac-harmonic map $(\phi,\psi)$ such that
 \begin{equation}\label{strong}
 (\phi_k,\psi_k)\to(\phi,\psi) \ \text{in} \  C_{loc}^\infty(M)
 \end{equation}
 \end{thm}
\begin{proof}
Define the energy concentration set to be
\begin{equation}
S:=\bigcap\limits_{r>0}\{x\in M|\varliminf\limits_{k\to\infty}E(\phi_k;B(x,r))\geq\varepsilon_0\}
\end{equation} 
where $\varepsilon_0$ is the positive constant in Theorem \ref{epsilonregularity} and $B(x,r)$ is the geodesic ball in $M$ with center at $x$ and radius $r$. Suppose $S$ is not empty, let us say $p\in S$. By the definition, we have
\begin{equation}\label{phik}
E(\phi_k;B(x_k,r_k))=\frac{\varepsilon_0}{2}
\end{equation}
with $x_k\to p$ and $r_k\to 0$ as $k\to\infty$. Let 
\begin{equation}
u_k(x)=\phi_k(x_k+r_kx), \ \ v_k(x)=r_k^{\frac12}\psi_k(x_k+r_kx).
\end{equation}
Then \eqref{phik} becomes 
\begin{equation}\label{uk}
E(u_k;B_1)=\frac{\varepsilon_0}{2}.
\end{equation}
 However,
\begin{equation}
E_\alpha(u_k;B_1):=\int_{B_1}|du_k|^{2\alpha}=r_k^{2\alpha-2}\int_{B(x_k,r_k)}|d\phi_k|^{2\alpha}\leq r_k^{2\alpha-2}\Lambda\to0
\end{equation}
as $k\to\infty$. This contradicts \eqref{uk}. Thus, $S$ is empty.

Now, for any point $x\in M$, there exist $r>0$ and a subsequence of $k\to\infty$ such that
\begin{equation}
E(\phi_k;B(x,r))<\varepsilon_0.
\end{equation}
Then by the $\varepsilon$-regularity Theorem \ref{epsilonregularity} and standard elliptic theory, we have 
\begin{equation}\label{Clestimate}
\|\phi_k\|_{C^l(B(x,r/4))}+\|\psi_k\|_{C^l(B(x,r/4))}\leq C
\end{equation}
for any $l>0$. Since $\{(\phi_k,\psi_k)\}$ has uniformly bounded energy, up to a subsequence if necessary, $\{(\phi_k,\psi_k)\}$ has a weak limit $(\phi,\psi)$. Then the regularity Theorem \ref{regularity} tells us that $(\phi,\psi)$ is a smooth $\alpha$-Dirac-harmonic map, and \eqref{Clestimate} implies \eqref{strong}.

\end{proof}


\nocite{*}


\bibliographystyle{amsplain}
\bibliography{reference}

\end{document}